\documentclass[11pt, amssymb,amsfonts]{amsart}
\usepackage{amsrefs}
\usepackage{algorithm}
\usepackage{algorithmic}
\usepackage{subcaption} 

\makeatletter
\def\l@section{\@tocline{1}{10pt plus0pt}{0pt}{}{\bfseries}}

\def\@tocline#1#2#3#4#5#6#7{\relax
    \ifnum #1>-1
  \ifnum #1>\c@tocdepth 
  \else
    \par \addpenalty\@secpenalty
    \begingroup \hyphenpenalty\@M
    \@ifempty{#4}{%
      \@tempdima\csname r@tocindent\number#1\endcsname\relax
    }{%
      \@tempdima#4\relax
    }%
    \parindent\z@ \leftskip#3\relax \advance\leftskip\@tempdima\relax
    \rightskip\@pnumwidth plus4em \parfillskip-\@pnumwidth
    #5\leavevmode\hskip-\@tempdima #6\nobreak\relax
    \hfil\hbox to\@pnumwidth{\@tocpagenum{#7}}\par
    \nobreak
    \endgroup
  \fi
\fi}

\makeatother

\usepackage{a4wide}
\usepackage[a4paper, margin=2.3cm]{geometry}
\usepackage{amsmath,amssymb,amsthm}
\usepackage{color}
\usepackage{parskip}
\usepackage{graphicx}
\usepackage{hyperref}
\usepackage{parskip}
\usepackage{setspace}
\usepackage{amsmath}
\usepackage{mathtools}
\usepackage{amssymb}
\usepackage{cases}
\usepackage{algorithm}
\usepackage{algorithmic} 
\usepackage{todonotes}
\usepackage{stmaryrd}
\usepackage{IEEEtrantools}
\usepackage[nocompress]{cite}
\usepackage{tikz}
\usetikzlibrary{patterns,snakes}
\newcommand{\norm}[1]{ \left|  #1 \right| }

\def\to{\rightarrow}

\def\E{{\hbox{\bf E}}}

 at 10 true pt

\def\rank{\hbox{\rm rank}}

\def\be#1{ \begin{equation}\label{#1} }

\def\bas{\begin{equation*}}
\def\eas{\end{equation*}}
\def\bi{\begin{itemize}}
\def\ei{\end{itemize}}

\def\emph#1{{\it #1}}
\def\textbf#1{{\bf #1}}

\theoremstyle{plain}
 \theoremstyle{plain}
  \newtheorem{theorem}{Theorem}
  \numberwithin{theorem}{section}

  \numberwithin{problem}{section}

  \newtheorem{lemma}{Lemma}
  
  \numberwithin{notation}{section}
  \numberwithin{lemma}{section}
  \newtheorem{corollary}{Corollary}
   \numberwithin{corollary}{section}

\theoremstyle{remark}
  \newtheorem{remark}{Remark}
  \numberwithin{remark}{section}

\theoremstyle{definition}
  \newtheorem{definition}{Definition}
\numberwithin{definition}{section}

\theoremstyle{definition}
  
  \numberwithin{summary}{section}

\DeclareUnicodeCharacter{2113}{\ensuremath{\ell}}
\DeclareUnicodeCharacter{2082}{\ensuremath{\infty}}
\DeclareUnicodeCharacter{221E}{\ensuremath{\infty}}
\begin{document}
\include{psfig}
\title[Regional Stability of Eigenvalues]{Regional Stability  and New Eigenvalue Perturbation Bounds}
\pagenumbering{arabic}

\author{Phuc Tran, Van Vu }
\thanks{trandangphuc234@gmail.com,   School of Engineering, VinUniversity (VinUni) \\ vuhavan001@gmail.com,  Department of Mathematics,  The University of Hong Kong (HKU)}
\date{}

\begin{abstract} 

Let \(A\) be an \(n\times n\) symmetric matrix with eigenvalues
\(
\lambda_1\geq\cdots\geq\lambda_n.
\)
Let
\(
\tilde A:=A+E,
\)
where \(E\) is a symmetric noise matrix, and denote the eigenvalues of \(\tilde A\) by
\(
\tilde\lambda_1\geq\cdots\geq\tilde\lambda_n.
\)
 Bounding the perturbation $|\tilde \lambda_i -\lambda_i| $ is a central 
problem in linear algebra and numerical analysis.  

\vskip2mm

In this paper, we prove new perturbation results by exploring the actual interaction between $E$ and the eigenvectors of $A$. In the setting where $E$ does not act adversarially with respect to these vectors (for instance, if $E$ is random),  we obtain a considerable improvement over Weyl's inequality.  One can routinely extend these results to the Hermitian and rectangular settings. 

\vskip2mm 

We obtain the new bounds as corollaries of a regional stability result, which provides a sufficient condition for a region
on the real line to be stable (containing the same number of eigenvalues) after 
the perturbation. This result is of independent interest.

\vskip2mm 
We prove our regional stability result using contour integral analysis. Our main new technical ingredient here 
is the {\it double-jump} argument, which is robust and could be useful in many other situations involving Neumann series.

\vskip3mm


\end{abstract} 
\maketitle

\section{Introduction} \label{sec: intro}

\subsection{The perturbation problem} Let $A$ and $E$ be symmetric real matrices of size $n$, and  $\tilde A:= A+E$. We view $A$ as the signal (or ground-truth) matrix and $E$ as the noise. 
Consider the spectral decomposition 
$$A =\sum_{i=1}^n \lambda_i u_i u_i^{\top} , $$ where $\lambda_i$ are the eigenvalues and $u_i$ the corresponding eigenvectors. We denote the eigenvalues in decreasing order, i.e.,  $\lambda_1 \ge \lambda_2 \geq \dots \geq \lambda_n$. We also order the singular values of $A$ in a similar manner 
$\sigma_1 \ge \sigma_2 \ge \dots \ge \sigma_n$. It is well known that 
$\{\sigma_1, \dots, \sigma_n \} = \{| \lambda_1| , \dots,|  \lambda_ n| \} $. 
In  particular,
$\sigma_1 = \|A\| = \max \{ \lambda_1, |\lambda_n |\} $ and $\sigma_n = \min_i \{ |\lambda_i | \}$.  
We use the notation $\tilde \lambda_i, \tilde u_i,$ etc., for $\tilde A$, with the same ordering.

 \begin{remark} \label{base}  We allow $A$ to have multiple eigenvalues. If $\lambda_i$ has multiplicity $m_i \ge 2$ and eigenspace $H_i$ of dimension $m_i$, then we choose any orthonormal basis in $H_i$ to be its eigenvectors. Our results hold regardless of which basis is chosen.  \end{remark} 

Eigenvalues and singular values are important in applications, and bounding their perturbations, $ | \tilde{\lambda}_p - \lambda_p | $
and $ | \tilde{\sigma}_p - \sigma_p | $, is a problem of fundamental interest.
The most frequently used tool for this purpose is  Weyl's inequality, which shows 
\begin{theorem}[Weyl \cite{We1, Book1}] \label{theo: weyl} For any $ 1 \leq p \leq n$, 
\begin{equation}
\norm{\tilde{\lambda}_p - \lambda_p} \leq \|E\|,\,\,\text{and}\,\,\norm{\tilde{\sigma}_p - \sigma_p} \leq \|E\|.
\end{equation}
\end{theorem}

This bound is sharp. For instance, consider $p=1$.
Assume that the first eigenvalue of $E$ equals its spectral norm (or largest singular value) $\| E \| $. Assume furthermore that $A$ and $E$ share the first eigenvector, namely,   $\lambda_1=u^\top A u$ and $ \lambda_1 (E)= \|E\|  = u^\top Eu = \| E \|$. Then, 
$$\tilde \lambda_1 \ge  u^\top \tilde A u  = u^\top A u + u^\top E u=  \lambda_1 + 
\| E \|. $$

Notice that in the sharpness demonstration above, $E$ interacts with the leading eigenvector  $u$ of $A$ in the worst possible way, namely, $u^\top E u = \| E\| = \max_{v, \| v\|=1 } v^\top E v$. For this to happen, the directions of the two vectors $u$ and $Eu$ must be the same. 
We observe that in many applications, $E$ does not act in such an adversarial way, and the two vectors $u$ and $Eu$ are typically in a general position, making $u^\top E u $ much less than $\| E \|$. As a matter of fact, two random directions, sampled uniformly from the unit sphere in high dimension,  are almost orthogonal.

The goal of this paper is to present a method to exploit this actual interaction to obtain spectral stability at scale much smaller than $\| E\|$. To quantify the idea, we fix an orthonormal eigenbasis $V$ of $A$ (see Remark \ref{base}), and define 
\begin{equation} \label{def: x} x:= \max_{u, v \in V} | u^\top E v |. \end{equation}

We will sometimes refer to $x$ as the {\it directional perturbation}. 
Our new eigenvalue bounds can be captured informally as follows: 
\begin{theorem}[Informal theorem] \label{informal} If $x $ is small compared to $\| E\|$, and the eigenvalues of $A$ do not cluster near $\lambda_p$, then $|\tilde \lambda_p - \lambda_p|$ is small. \end{theorem}

There have been various refinements of Weyl's theorem, e.g.,
\cite{Ma2023, DopicoMoroMolera2000, SS1, Kato1, OVK13, EisenstatIpsen1998,Nakatsukasa2012, Stewart1984, BoutsikasDrineasIpsen2024, Rump2009, DexterBoutsikasMaIpsenDrineas2025, KY1, CDF1, IpsenNadler2009}.
The nature of these refinements, however, are  different from ours. 
We discuss several of them in Subsection~\ref{sec: comparison}. The method we developed here, which relies on the regional stability and the double-jump argument to analyze matrix series, also seems quite different from previous proofs.

\noindent\textbf{Notation.} We use standard asymptotic notation such as $O, o, \Theta, \Omega, \tilde{O},$ etc., under the assumption that $n \rightarrow  \infty$.  For a matrix $M$, $\|M \| $ denotes its operator (spectral) norm and $\Lambda(M)$ denotes its spectrum. Logarithms have a natural base. 

\subsection{Regional stability}

The key concept in our study is  {\it regional stability}. We are going to derive the stability of individual eigenvalues from that of a region. In this discussion, we fix $A$ and $E$. 

A region on the real line is the union of a finite number of disjoint open intervals. Let  $D$ be a region and $\Lambda_D (A)$ the set of eigenvalues of $A$ in $D$.  We say that $D$ is {\it stable} (with respect to $E$) if $| \Lambda_D(A)| = | \Lambda_D (\tilde A)|$. In other words, the number of eigenvalues in $D$ remains the same after the perturbation. 
We remark that the stability of the inertia of the matrix is equivalent to the stability of both  $D_{-} = (-\infty, 0) $ and $D_{+}=(0,+\infty)$. \footnote{The inertia of a symmetric matrix $A$ is the triple
\(
\operatorname{In}(A)=(n_+(A),n_-(A),n_0(A)),
\)
where $n_+(A)$, $n_-(A)$, and $n_0(A)$ denote the numbers of positive, negative, and zero eigenvalues of $A$, respectively; see \cite{HJBook}.}

Let $\partial D$ be the boundary of $D$. We define 
\begin{equation} \label{def: deltaD}
\textstyle   \delta_D := \min_{ 1 \le i \le n } 
\mathrm{dist} (\lambda_i, \partial D) .  
\end{equation}
Thus, $\delta_D$ is the {\it minimum distance} (or {\it gap})  between the boundary $\partial D$ and the eigenvalues of $A$.  Weyl's inequality implies the following.
\begin{corollary}[Weyl criterion] \label{cor:stable}
If $D$ is a region such that $\delta_D > \| E \|$, then $D$ is stable. 
\end{corollary}

\subsection{Our new stability theorem} \label{subsec: intro stab theo}

Let $K > 0$ be a parameter to be chosen and let $r_K$ denote the number of eigenvalues of $A$, counted with multiplicity, whose distance from $D$ is at most $K\|E\|$. Our regional stability theorem asserts that for $D$ to be stable, it is sufficient to have 
$$
\delta_D  \gtrsim_{\log} \max \bigg \{  \frac{\| E \|} {K}, \sqrt { r_K \cdot x \|E\| }  \bigg\}.
$$


 The spirit of this result is similar to that of (Informal) Theorem \ref{informal}. If  $x$ is small (compared to $\|E\|$) and the eigenvalues do not cluster near $D$, then it is possible to have both terms in the right-hand side significantly smaller than $\| E \|$, improving Corollary \ref{cor:stable}.

We are now ready to state the formal theorem. 
Let $C_D$ denote the number of connected components of $D$. Recall that
\[
x:=\max_{u,v\in V}|u^\top E v|,
\]
where $V$ is an orthonormal eigenbasis of $A$.
\begin {theorem} \label{theo: main0} Let  $D$ be a region. Assume that 
for some parameter $\textstyle K > 0$,
$$\delta_D \ge 45 \log \left(3+\frac{\sigma_1}{\delta_D}  \right) \cdot C_D \cdot  \max \bigg \{  \frac{\| E \|} {K}, \sqrt { r_K \cdot x \|E\| }  \bigg\}.$$ Then $D$ is stable. 
\end{theorem}

Here and later we allow $K$ to depend on $n$. 
One can make the constants  $45$ and $3$ smaller by optimizing the 
constants in Section \ref{sec: proof} and Appendix \ref{section: tech estimates}.  
However, we make no attempt to do so. A discussion in Section 2 shows that this bound is sharp up to the logarithmic factor.

To derive new perturbation bounds for eigenvalues, we use  Theorem \ref{theo: main0} with a properly chosen region. For instance, if we want a lower bound for  $\tilde \lambda_p $, we can consider
the region $D = (\lambda_p - T, +\infty)$, for a parameter
 $ 0 < T < \lambda_p - \lambda_{p+1} $ to be chosen.  Then $D$ contains exactly the first 
 $p$ eigenvalues of $A$. If $D$ is
stable, it also contains exactly the first $p$  eigenvalues of $\tilde A$, and hence 
\[
\tilde \lambda_p  \ge \lambda_p - T.
\]
We then optimize  $T$ under the constraint on $\delta_D$  in Theorem \ref{theo: main0}. Similarly, for an upper bound, consider
$ D = (\lambda_p  + T, +\infty)$. For details, see Section \ref{section: new pertubration bounds}. 

We can also derive a bound for the least singular value as follows. Assume that $\sigma_ n > T > 0$. Consider the region $D =(-T,T)$.
This region is empty (with respect to $A$). If $D$ is stable, then it is also empty with respect to $\tilde A$ as well, which means $\tilde \sigma_n > T$.

   \begin{remark}[A relation to pseudospectrum]
This notion of regional stability is related to, but distinct from, the classical notion of pseudospectrum. For a matrix $A$ and $\epsilon>0$, the $\epsilon$-pseudospectrum of $A$ is defined by
\[
\Lambda_\epsilon(A)
:=
\bigcup_{\|E\|<\epsilon}\Lambda(A+E).
\]
Thus, the $\epsilon$-pseudospectrum describes the possible locations of individual eigenvalues under an arbitrary perturbation of norm at most $\epsilon$, whereas our notion of stability indicates whether the \emph{number} of eigenvalues in a prescribed region $D$ is preserved under a given perturbation. Consequently, the study of $\Lambda_{\epsilon} (A)$ is inherently a worst-case analysis.  On the other hand, in our study, the key idea is to explore the actual interaction of $E$ with the eigenvectors of $A$.

\end{remark}

\subsection{Stability of a projection and the double-jump argument} \label{subsec: main theorem intro}
In this subsection, we describe the main ideas behind the proof of Theorem~\ref{theo: main0}, focusing in particular on the double-jump argument, which is our key new technical ingredient. 

 Let $H_D$ (respectively, $\tilde H_D$) be the subspace spanned by the eigenvectors of $A$ (respectively, $\tilde A$) corresponding to eigenvalues in $D$. Let $\Pi_D$ and $\tilde\Pi_{D}$ denote the corresponding orthogonal projections; see Remark~\ref{base}. Observe that if 
\[
\|\tilde\Pi_{D}-\Pi_{D}\|<1,
\]
then, by the pigeonhole principle, $H_D$ and $\tilde H_D$ must have the same dimension. Consequently, $D$ is stable. 

The main technical step is to bound $\|\tilde\Pi_{D}-\Pi_{D}\|$. 
We use contour analysis for this purpose, making use of a contour integral representation of the difference 
$\tilde\Pi_{D}-\Pi_{D}$. 
Indeed, by the Riesz projection formula, we can write
$$
\tilde\Pi_D-\Pi_D
=
\frac{1}{2\pi\mathbf{i}}
\int_\Gamma
\left((z-\tilde A)^{-1}-(z-A)^{-1}\right)\,dz,
$$
where \(\Gamma\) is a suitably constructed contour enclosing precisely the eigenvalues of \(A\) and \(\tilde A\) lying in \(D\). For example, if \(D\) is a finite interval, we can simply take \(\Gamma\) to be a rectangle whose vertical edges intersect the real axis at the endpoints of \(D\). As $\delta_D >0$, we can always assume that the endpoints of $D$ are not eigenvalues of $A$. Here and throughout, \(\mathbf{i}:=\sqrt{-1}\).

This representation is well-known in functional and numerical analysis; see, for instance, \cite{Kato1, HJBook}. 
Using the fact that $\tilde A =A+E$ and the resolvent formula, we can, assuming convergence,  
write 
\(
(z- \tilde A)^{-1}-(z-A)^{-1} = \sum_{s=1}^{\infty}(z-A)^{-1} [ E(z-A)^{-1}]^{s},  
\)
and hence
\begin{equation}  \label{id: Gs Pi}
    \tilde\Pi_{D}-\Pi_D= \frac{1}{2\pi \textbf{i}} \int_{\Gamma }\sum_{s=1}^{\infty}  G_s(z) \, dz, \text{where}\,\, G_s(z):= (z-A)^{-1} [ E(z-A)^{-1}]^{s}.
\end{equation}

To bound the series $\sum _{s=1}^{\infty} G_s(z)$, the usual argument is to bound $\| \frac{ G_{s+1}(z)  }{G_s(z)} \|$, which is 
$\| E (z-A)^{-1} \|$.  As $z$ runs on $\Gamma$, $\max_{z \,\text{on}\, \Gamma}\| (z-A)^{-1} \|$ is precisely 
$\delta_D^{-1}$, and we obtain the standard bound $\frac{ \|  E \| }{\delta_D }$ for the ratio in question. In the case when $E$ is random, the bound $\frac{ \|  E \| }{\delta_D }$ is sharp, up to a constant factor (see Lemma \ref{Wigner property}).
One can use this bound to derive Corollary \ref{cor:stable} and  (a version of) the classical Davis--Kahan theorem
(for perturbation of projections). There are many refinements of the Davis--Kahan theorem (e.g., \cite{tran2025davis, YWS1, SS1, KX1}), but they all need the assumption that $\delta_D \ge \| E \|$, and thus cannot be used to improve Corollary \ref{cor:stable}.

The basic idea 
of our double-jump strategy is as follows:  we split the series $\sum_{s=1}^ {\infty} \frac{1}{2\pi \textbf{i}} \int_{\Gamma }  G_s(z)\, dz$ into two parts, one containing the odd-indexed terms and the other the even-indexed terms (thus the indices jump by two in each series). The new ratio now is  $$\bigg\| \frac{ G_{s+2}(z) }{G_s(z)} \bigg\| = \| E(z-A)^{-1} E (z-A)^{-1} \|. $$ Our key observation is that this can be smaller than $1$, even if the single-jump $\| E (z-A)^{-1}\| $ is much larger than $1$. The point here is that the matrix $M:= E(z-A)^{-1}$ is no longer symmetric, so $\| M^2\|$ can be much smaller than $\| M \|^2$. 
 
 In bounding  $\| E(z-A)^{-1} E (z-A)^{-1} \|$, we have found a way to exploit the parameter $x$ and information on the density of the spectrum near $D$. In particular, we  will be able to show 
 (see Lemma \ref{lem: double-jump}) that 
    \begin{equation} \label{bounding E(z-A)E(z-A)}
   \textstyle\| E(z-A)^{-1} E (z-A)^{-1}\| \leq h :=    \frac{r_{K}\cdot x \|E\|}{\delta_D^2} + \frac{2\|E\|}{\delta_D K} + \frac{1}{K^2}. 
\end{equation}
 Combining \eqref{bounding E(z-A)E(z-A)} and some technical estimates, we prove the following theorem.
 %

\begin{theorem}[Stability of projections] \label{theorem: main}
    Assume that for some $ K >0$,
$$ \delta_D \geq 6 \max \bigg\{\frac{\|E\|}{K}, \sqrt{r_K \cdot x \|E\|}   \bigg\}.$$ 
\begin{itemize}
    \item If $\delta_D \geq \|E\|$, then 
    \begin{equation} \label{main theorem: bound delta>E}
   \| \tilde\Pi_{D} - \Pi_D\| \leq 44 \log \left(3+\frac{\sigma_1}{\|E\|}  \right) \cdot  C_D \cdot \left(\frac{ r_K \cdot x}{\delta_{D}}  +   \frac{1}{K} \right).      
    \end{equation}
    \item If $\delta_D < \|E\|$,  then
    \begin{equation} \label{main theorem: bound delta<E}
         \| \tilde\Pi_{D} - \Pi_D\| \leq 44 \log \left(3+\frac{\sigma_1}{\delta_D}  \right) \cdot  C_D \cdot \left( \frac{\|E\|}{K \delta_D} + \frac{r_K\cdot x \|E\|}{\delta_D^2} \right).
    \end{equation}
\end{itemize}

\end{theorem}
 To derive  Theorem~\ref{theo: main0},  we only need to use the case 
$\delta_D <\|E\|$. (The bound in the other case is still interesting in its own right.) 

The gap condition in Theorem~\ref{theo: main0},
\begin{equation}\label{gap condition 0} \textstyle
\delta_D
\geq
45\log\left(3+\frac{\sigma_1}{\delta_D}\right)\cdot C_D \cdot
\max\big\{
\frac{\|E\|}{K},
\sqrt{r_K \cdot x\|E\|}
\big\},
\end{equation}
clearly implies the assumption of Theorem~\ref{theorem: main}. Hence, by \eqref{main theorem: bound delta<E},
$$
\|\tilde\Pi_D-\Pi_D\|
\leq
44\log\left(3+\frac{\sigma_1}{\delta_D}\right) \cdot C_D \cdot
\left(
\frac{\|E\|}{K\delta_D}
+
\frac{r_Kx\|E\|}{\delta_D^2}
\right).
$$
By \eqref{gap condition 0},
\(
44\log\left(3+\frac{\sigma_1}{\delta_D}\right)C_D
\frac{\|E\|}{K\delta_D}
\leq\frac{44}{45},
\)
while
\(
44\log\left(3+\frac{\sigma_1}{\delta_D}\right)C_D
\frac{r_Kx\|E\|}{\delta_D^2}
\leq
\frac{44}
{45^2 C_D\log\left(3+\frac{\sigma_1}{\delta_D}\right)}
\leq
\frac{44}{45^2\log 3}.
\)
Therefore,
$$ 
\|\tilde\Pi_D-\Pi_D\|
\leq
\frac{44}{45}+\frac{44}{45^2\log 3}
<1.
$$
This implies that \(D\) is stable, proving Theorem~\ref{theo: main0}.

\begin{remark} All of our results hold for
the Hermitian setting. It is possible to refine Theorem \ref{theorem: main}, given more detailed information about the spectrum of $A$; see Section \ref{sec: remark and extension}. 

\end{remark}

 \subsection{Structure of the rest of the paper}
\begin{itemize}
\item In Section~\ref{section: new bound random E}, we present new eigenvalue perturbation bounds when \(E\) is random, an important setting in many applications. Our model for the random matrix $E$ is very general, allowing an arbitrary variance profile of the entries. Next,  we  demonstrate that our results 
are sharp up to a logarithmic factor. We will also  discuss a connection with results in random matrix theory.

\item In Section~\ref{section: new pertubration bounds}, we apply our stability theorem (Theorem~\ref{theo: main0}) to derive deterministic perturbation bounds for eigenvalues. These results imply the previous results for random $E$ as special cases.  At the end of this section (Subsection \ref{sec: comparison}), we discuss related perturbation results from matrix theory. 

\item In Section~\ref{sec: proof}, we prove Theorem~\ref{theorem: main} on the stability of spectral projections. This is the main technical part of the paper.

\item Section \ref{sec: remark and extension} contains several remarks. We will present a refinement of Theorem \ref{theorem: main}, in the case when we have more information about the distribution of the eigenvalues of $A$.  
We also discuss a refined definition of the parameter $x$ and the extension of our results for Hermitian matrices. 

\item The appendices contain proofs of several technical results,  deferred from the main text in order to maintain the flow of the presentation.

\end{itemize}

\textit{Acknowledgments.}
The authors would like to thank the anonymous reviewers for their valuable comments and suggestions, which greatly helped improve the presentation and quality of the paper.

\section{New eigenvalue perturbation bounds for random $E$} \label{section: new bound random E}

While our general eigenvalue perturbation bounds are deterministic, in this section we focus on the case where \(A\) is deterministic and \(E\) is random. Random matrices provide a natural and widely used model for noise. Moreover, in the random setting, we can often obtain sharp estimates for both \(\|E\|\) and \(x\), which allow us to derive simpler and more explicit bounds than those in the general deterministic setting in Section~\ref{section: new pertubration bounds}.

We use the following model of random matrices, in which we allow the entries to have arbitrary variances.

\begin{definition} [Sub-Gaussian random variable]
A random variable $\xi$ is sub-Gaussian  if there is a constant $C >0$ such that 
\(
\|\xi\|_{\psi_2}:= \inf\big\{c > 0: \E \left[\exp \left(\frac{\xi^2}{c^2} \right) \right] < 2 \big\} \leq C.
\)
\end{definition}
\begin{definition}[Sub-Gaussian  matrix] \label{def: random E}
We say that a symmetric random matrix \(E=(\xi_{ij})_{i,j=1}^n\) is a \emph{sub-Gaussian matrix}  if the upper-triangular entries
$(\xi_{ij})_{1\leq i\leq j\leq n}$ are independent, centered sub-Gaussian random variables such that 
\(
C_1 \le \mathbf E\xi_{ij}^2\leq 1,
\)
and
\(
\max_{i,j}\|\xi_{ij}\|_{\psi_2}\leq C_2
\)
for some absolute constants $C_1, C_2 >0$.

\end{definition}

Our model does not require the entries to be identically distributed or imposes any regularity assumption on the variance profile. It contains several important classes of random matrices arising in both theory and applications; see, for instance, \cite{AchlioptasMcSherry2007, OVK13, CP1, Dwok2, TranVishnoiVu2025, JSS1, AKS1}. In particular, after normalization,  it contains GOE, and,  more generally, any Wigner matrix. 

We will repeatedly use the following well-known estimates of $\|E\|$ and $x$; see \cite{Ver1book, tao2012topics, van2017spectral, OVK13}.
\begin{lemma}\label{Wigner property}
If $E$ is a sub-Gaussian matrix, then, with probability $1-o(1)$,
   \begin{equation}
    (2\sqrt{C_1}+o(1)) \sqrt{n} \leq \|E\| \leq (2+o(1)) \sqrt n,
    \label{eq:wigner-norm}
\end{equation}
and
\begin{equation}
    x=O(\sqrt{\log n})=o(\log n).
    \label{eq:wigner-x}
\end{equation}

Furthermore, there is a constant $c>0$ (depending on $C_1, C_2$) such that for any fixed unit vector $u$, with probability $1-o(1)$, $\| E u \| \ge c n^{1/2}$.
\end{lemma}
The significant difference in scale between $\|E \|$ and $x$ is a source of improvement, as discussed in Theorem \ref{informal}. 

A parameter that will play an important role is the gap between consecutive eigenvalues.
\begin{definition}[Eigenvalue gaps] \label{gapdef}
We set  $\delta_p := \lambda_p -\lambda_{p+1} $, for $p=1,2, \dots, n-1$. Let $\delta_{(p)}$ be the distance from $\lambda_p$ to the closest eigenvalue. Thus,   $\delta_{(p)} := \min \{ \delta_p, \delta _{p-1} \}$, for $2 \le p \le n-1$, $\delta_{(1)}= \delta_1$ and $\delta_{(n)} = \delta_{n-1}$. Finally, let   $\delta_{\rm min} 
:= \min_{1 \le p \le n-1} \delta_p $ be the minimum gap. 
\end{definition}

\subsection{The perturbation of the leading eigenvalue}
In this subsection,  we focus on the leading eigenvalue $\lambda_1$. 
Similar results can be obtained for any $\lambda_p$; see Subsection \ref{subsec: leading}.

Assume that we aim to improve Weyl's bound \(\|E\|\) to \(\|E\|/K\), where \(K > 0\) may depend on \(n\) (e.g., \(K=n^{1/8}\)). Our theorems below show that such an improvement is possible, up to a logarithmic factor, at the cost of an additional error term depending on the number of eigenvalues near \(\lambda_1\). This term remains small when the spectrum does not cluster near \(\lambda_1\); see Theorem \ref{informal}.



\begin{theorem} \label{theoremtoy1}
 Let $E$ be a sub-Gaussian matrix. Fix a parameter $ K= K(n) > 0 $ and let $r$ be the number of eigenvalues greater than $\lambda_1 - 3K \sqrt n$. Then, with probability $1-o(1)$, 
$$ \tilde{\lambda}_1 - \lambda_1 \le 
135  \log(3+\sigma_1) \cdot \max  \bigg\{ \frac{\sqrt{n}}{K}, 
   \sqrt {r \log n } \cdot n^{1/4} \bigg\}.  $$
\end{theorem}
%
%
%
Ignoring the log terms, the right-hand side is essentially 
$ \max  \big\{ \frac{\sqrt{n}}{K}, 
   \sqrt {r \log n } \cdot n^{1/4} \big\}$. This improves 
Weyl's bound ($O(\sqrt n)$) if we can choose $K \gg 1$ such that 
$r =o(\sqrt n)$.

One can guarantee that $r$ is relatively small if the minimum gap between the eigenvalues is reasonably large. 
For example, set $K=n^{1/8}$. If  $\delta_{\min}\geq n^{3/8}$, then
the number of eigenvalues greater than $\lambda_1 - 3K\sqrt{n}$ is at most $\frac{3n^{1/2+1/8}}{n^{3/8} } = 3n^{1/4}$. We obtain the following corollary. 
\begin{corollary}  \label{mingap1} Let $E$ be a sub-Gaussian matrix and assume that $\delta_{\rm min} \ge n^{3/8}$. Then, with probability $1-o(1)$, 
$$\tilde \lambda_1 - \lambda_1 = \tilde O( n^{3/8}) . $$
\end{corollary}
If the maximum multiplicity of any eigenvalue is at most $m$ and the gap between any two different eigenvalues is at least $\delta$, then one can still use the same counting argument to bound $r$, at the cost of an extra factor $m$.

The treatment of the lower bound is more delicate. We need to take into account the gap $\delta_1 =\lambda_1 -\lambda_2$. Similar to Theorem \ref{theoremtoy1}, we fix a parameter $\ K=K(n) > 0 $ and let $r$ be the number of eigenvalues 
greater than $\lambda_1 - \frac{\delta_1}{2} - 3K \sqrt n$. Set 
\[\textstyle
 L := 270 \log(3+\sigma_1) \cdot  \max  \big\{ \frac{\sqrt{n}}{K}, 
   \sqrt {r \log n } \cdot n^{1/4} \big\}.
\]

\begin{theorem} \label{theoremtoy2} Let $E$ be a sub-Gaussian matrix. Assume that $\delta_1 \ge L$.
Then, with probability $1 -o(1)$,
$$ \tilde \lambda_1 -\lambda_1 \ge - L/2,\,\,\text{or equivalently},\,\,\lambda_1 - \tilde{\lambda}_1   \le L/2 .$$ 
    \end{theorem}
Similar to Corollary \ref{mingap1}, we obtain the following corollary. 
\begin{corollary} \label{mingap2}
    Let $E$ be a sub-Gaussian matrix and assume that $\delta_{\rm min} \ge n^{3/8} $. Then, with probability $1-o(1)$, 
$$ \lambda_1 - \tilde \lambda_1 =\tilde O( n^{3/8}). $$
\end{corollary}
\begin{remark}[Small gaps]
An important feature of our bounds is that they remain effective even if the gap \(\delta_1\) is significantly smaller than the noise level $\|E\|$.
\end{remark}

\subsection{The sharpness of our results}     In the rest of this section, we demonstrate the strength of our results by an example showing that the above bounds are sharp up 
     to the logarithmic factor.

Assume that $\lambda_1 >0$ and $A$ has at most $r_0$ positive eigenvalues, where $\lambda_1$ and $r_0$ satisfy 
 \begin{equation} \label{assumptionr0} \textstyle \sqrt{n} \le \lambda_1 \le \frac{n^{3/4}}{\sqrt{r_0}}. \end{equation}
If we set $K = \frac{\lambda_1}{4\sqrt{n}}$, then the parameter $r$ in 
Theorem \ref{theoremtoy1}  is at most $r_0$.  Consequently, Theorem \ref{theoremtoy1}, together with the above value of $K$ and assumption \eqref{assumptionr0}, yields
\begin{equation} \label{sharpness1}
\textstyle
\tilde{\lambda}_1 -\lambda_1= \tilde O(  \max  \big\{ \frac{\sqrt{n}}{K}, 
   \sqrt {r \log n } \cdot n^{1/4} \big\}) \leq \tilde O(  \max  \big\{ \frac{\sqrt{n}}{K}, 
   \sqrt {r_0 \log n } \cdot n^{1/4} \big\}) =
\tilde{O}\big(\frac{n}{\lambda_1} \big).
\end{equation}

 We  show that this bound is sharp, up to the logarithmic factor hidden in $\tilde O(1)$. To see this, suppose, for simplicity, that 
$\tilde \lambda_1 = \| A +E\|$. Then
\[
    \tilde\lambda_1
    = \|A+E\|
    \geq \|(A+E)u_1\|
    = \|\lambda_1 u_1+Eu_1\|.
\]
Consequently,
\[
    \tilde\lambda_1^2
    \geq
    \lambda_1^2
    +2\lambda_1 u_1^\top Eu_1
    +\|Eu_1\|^2.
\]
By Lemma \ref{Wigner property} we have, with probability $1-o(1)$,
\[
    |u_1^\top Eu_1|=O(\sqrt{\log n}),
    \qquad
    \|Eu_1\|^2=\Theta(n),
    \qquad
    \|E\|^2=\Theta(n).
\]

\vskip2mm 
By assumption \eqref{assumptionr0}, 
\(
    \lambda_1\sqrt{\log n}=o(n),
\)
which is negligible. Thus, we have, with probability $1-o(1)$, 
$\tilde \lambda_1^2 = \lambda_1^2 + \Theta (n) $. As $\lambda_1 \ge \sqrt n$, it follows that $\tilde \lambda_1 -\lambda_1 = \Theta ( \frac{n}{\lambda_1} ),$
which shows that \eqref{sharpness1} is sharp up to the logarithmic term hidden in $\tilde O(1)$.

\subsection{Low-rank setting} \label{subsec: low-rank discussion} \label{subsec: lowrank}
Concerning the {\it non-clustering} property mentioned in Theorem \ref{informal}, let us consider the case when $A$ has low rank. In this setting, non-clustering is guaranteed automatically
in regions away from zero, as there are only a few nonzero eigenvalues. 

This special case is important for several reasons. Low-rank structure arises naturally in many applications, including matrix completion and Gaussian mixture models; see, e.g., \cite{CP1, CR1, KSV1, ZZ1}. Moreover, real-world data matrices are often observed to be low-rank or approximately low-rank, and several theoretical works seek to explain this phenomenon; see, e.g., \cite{UT1, ChamberlainRothschild1983}. Finally, the model \(A+E\), with \(A\) low-rank and \(E\) random, has been studied extensively in random matrix theory, particularly in the context of deformed Wigner matrices; see \cite{B-GN1, B-GGM1, CDF1, PRS1, KY1, altschuler2026spectral} and references therein.

Using our stability theorem for $K=\frac{|\lambda_p|}{2\|E\|}$, we obtain the following result. The detailed proof is deferred to  Appendix~\ref{sec: proof of deformWig}.
\begin{theorem} \label{cor: simpleWigeruplow}
    Assume that  $\max \{\rank\,A, \log \| A \| \}  \le \log^c n$, for some constant $c >0$. 
    There are constants $c_1,  c_2$ depending on $c$ such that the following holds. If $E$ is a sub-Gaussian matrix and for some index $1\le p \le n$, 
    $\textstyle \min\{\delta_{(p)}, |\lambda_p|\}  \geq  \log^{c_1} n \cdot \max\{\frac{n}{|\lambda_p|}, n^{1/4} \},$ then, with probability $1- o(1)$, 
    $$ |\tilde{\lambda}_p -\lambda_p| \leq  \log^{c_2}n\cdot \max \bigg\{\frac{n}{|\lambda_p|}, n^{1/4} \bigg\}.$$ 
    \end{theorem}

For comparison, we note that Weyl's bound (together with Lemma \ref{Wigner property}) gives, with probability $1-o(1)$, that 
$$|\tilde \lambda_p - \lambda_p| \leq \| E \| \leq (2+o(1)) \sqrt n. $$ 
Thus, we obtain an improvement if  $|\lambda_p|$ is relatively large compared to $\|E \|$.  For instance, if $|\lambda_p| \ge n^{2/3}$, then Theorem \ref{cor: simpleWigeruplow} gives 
\[ \textstyle
|\tilde \lambda_p -\lambda_p | \le  \log^{c_2} n \cdot n^{1/3}.
\]

\subsection{ A connection to random matrix theory.} Theorem \ref{cor: simpleWigeruplow} extends well-known results in random matrix theory. First, observe that if  $0<\lambda_p \le  n^{3/4}$, then our bound simplifies to
\(\textstyle
|\tilde{\lambda}_p-\lambda_p|
=
\tilde O\left(\frac{n}{\lambda_p}\right).
\)
Up to a polylogarithmic factor (hidden in $\tilde O$), this bound extends results concerning the outliers of a deformed Wigner matrix, which have been studied extensively in random matrix theory; see \cite{CDF1, KY1, PRS1, B-GGM1, B-GN1} and references therein. For instance, Benaych-Georges and Nadakuditi proved: 
\begin{theorem}[{\cite[Theorem 2.1]{B-GN1}}]\label{BGN}
Let $E$ be GOE and $A$ have constant rank.
If $\lambda_1=2c\sqrt{n}$ for some constant $c>1/2$, then, with probability $1-o(1)$,
\[ 
\tilde{\lambda}_1-\lambda_1
=
(1+o(1))\frac{n}{\lambda_1}.
\]
\end{theorem}
Our result shows that the upper bound
\(
\frac{n}{\lambda_1}
\)
continues to hold, up to a logarithmic factor, even when \(\lambda_1/\sqrt n\to\infty\) (with \(\lambda_1/\sqrt n\) as large as \(n^{1/4}\)). Furthermore, our result applies in the substantially more general setting of Definition~\ref{def: random E}, where \(E\) need not be GOE or even a Wigner matrix. On the other hand, it is not clear if the precise asymptotic result of Theorem~\ref{BGN} continues to hold in this general setting.

\subsection{The least singular value}
In applications, extremal singular values often play an important role. In particular, there is an extensive literature on the least singular value of randomly perturbed matrices; see, e.g., \cite{JSS1, JSS2, TV1, TV2, sankar2006smoothed, bourgain2017problem, farrell2016smoothed, livshyts2021smallest, rudelson2008littlewood, tikhomirov2020invertibility, DexterBoutsikasMaIpsenDrineas2025}.

 Analogously to our results for the leading eigenvalue, we obtain an improvement over Weyl's bound for the least singular value
\[
\sigma_n=\min_{1\leq i\leq n}|\lambda_i|.
\]
Fix $ K=K(n) > 0$, and let $r$ be the number of singular values of $A$ smaller than $\sigma_n+3K\sqrt n$. Set
\[ \textstyle
L
:=
270\log(3+\sigma_1) \cdot   \max  \big\{ \frac{\sqrt{n}}{K}, 
   \sqrt {r \log n } \cdot n^{1/4} \big\}.
\]
We show that if $\sigma_n$ is not too small (but could still be much less than $\| E\|$), and the singular values do not cluster near $\sigma_n$, then $\tilde \sigma_n$ and $\sigma_n$ are close. In particular, they are of the same order of magnitude. 
 \begin{theorem}\label{theoremtoy3}
Let $E$ be a sub-Gaussian matrix and suppose that $\sigma_n\geq L$. Then, with probability $1-o(1)$,
\[
\tilde\sigma_n
\geq
\sigma_n-L/2
\geq \sigma_n/2.
\]
\end{theorem}
%








\section {Deterministic   perturbation bounds for eigenvalues} \label{section: new pertubration bounds}

In this section, we follow the general strategy outlined in Subsection~\ref{subsec: intro stab theo} and use Theorem~\ref{theo: main0} to derive new deterministic perturbation bounds for eigenvalues (Subsection~\ref{subsec: leading}) and the least singular value (Subsection~\ref{section: least}). Applying these deterministic results to sub-Gaussian Wigner-type noise yields Corollaries~\ref{cor: lambdaPWigner} and \ref{cor: lambdapWignerupper}, which extend the representative results from the previous section to an arbitrary \(p\). In the final subsection, we discuss several previous results in the literature.

\subsection{Perturbation bounds for eigenvalues} \label{subsec: leading}
Consider the eigenvalues
\(
\lambda_1\geq \lambda_2\geq \cdots \geq \lambda_n,
\)
and fix an index \(1\leq p\leq n\). We seek bounds on \(\tilde\lambda_p\). We follow the general strategy described at the end of subsection \ref{subsec: intro stab theo}. 

%

\subsubsection{Lower bounds} \label{subsec:lowerboundLeading} Consider $D=(\lambda_p-T, +\infty)$. If $T < \delta_p$, then $\delta_D= \min\{T, \delta_p -T\}$. For a  parameter $K=K(n) > 0$, let $r_{p,K}$ be  the number of eigenvalues at least $\lambda_p- \frac{\delta_p}{2} - K\|E\|$.
Assume that 
\begin{equation} \label{ass: deltap}
    \delta_p \geq 90 \log \left(3+\textstyle\frac{\sigma_1}{\sqrt { r_{p,K} \cdot x \|E\|}} \right) \cdot  \max \bigg \{    \frac{\| E \|} {K}, \sqrt { r_{p,K} \cdot x \|E\| }  \bigg\},
\end{equation}
and set 
\begin{equation} \label{set: T for p lower}
    T :=45 \log \left(3+\textstyle\frac{\sigma_1}{\sqrt { r_{p,K} \cdot x \|E\|}}  \right) \cdot \max \bigg\{  \frac{\| E \|} {K},  \sqrt {r_{p,K} \cdot  x \|E\| } \bigg\}.
\end{equation}
We verify that with this choice of $T$, the assumption of Theorem \ref{theo: main0} is satisfied. Indeed, by \eqref{ass: deltap} and \eqref{set: T for p lower}, $T \leq \delta_p/2$ and hence $\delta_D=T$. Thus, any eigenvalue within distance \(K\|E\|\) of \(D\) is at least $\lambda_p-T-K\|E\|
\geq
\lambda_p-\frac{\delta_p}{2}-K\|E\|, $
so \(r_K\leq r_{p,K}\).
Also by \eqref{set: T for p lower}, $\sqrt {r_{p,K} \cdot  x \|E\|} \leq T =\delta_D$. These observations imply
\begin{equation} \label{verify condition for lower bound}
    \begin{split}
        \delta_D=T&= 45 \log \left(3+\textstyle\frac{\sigma_1}{\sqrt { r_{p,K} \cdot x \|E\|}}  \right) \cdot \max \bigg\{  \frac{\| E \|} {K},  \sqrt {r_{p,K} \cdot  x \|E\| } \bigg\} \\
        &\geq 45 \log \left(3+\frac{\sigma_1}{\delta_D}  \right) \cdot \max \bigg\{  \frac{\| E \|} {K},  \sqrt {r_{p,K} \cdot  x \|E\| } \bigg\} \\
        &\geq 45 \log \left(3+\frac{\sigma_1}{\delta_D}  \right) \cdot \max \bigg\{  \frac{\| E \|} {K},  \sqrt {r_{K} \cdot  x \|E\| } \bigg\}.
    \end{split}
\end{equation}
     %
Thus, Theorem~\ref{theo: main0} applies, yielding the following theorem.
\begin{theorem} \label{theo: leadinglower}
    Assume that for some parameter  $K > 0,$  
    \begin{equation} \label{mainassumption0}
          \delta_p \geq 90 \log \left(3+\textstyle\frac{\sigma_1}{\sqrt { r_{p,K} \cdot x \|E\|}} \right) \cdot  \max \bigg \{    \frac{\| E \|} {K}, \sqrt {  r_{p,K} \cdot x \|E\| }  \bigg\}.
    \end{equation}
    Then 
    \[ 
    \tilde{\lambda}_p \geq \lambda_p - 45 \log \left(3+\textstyle\frac{\sigma_1}{\sqrt { r_{p,K} \cdot x \|E\|}}  \right) \cdot  \max \bigg \{   \frac{\| E \|} {K}, \sqrt {  r_{p,K} \cdot x \|E\| }  \bigg\}.
    \]
\end{theorem}
If $E$ is a sub-Gaussian matrix, then by Lemma \ref{Wigner property}, with probability $1-o(1)$, we can replace $\|E\|$ and $x$ by the upper bounds $3\sqrt{n}$ and $ \log n$ respectively. As a result, together with the fact that $r_{p,K} \cdot \log n \cdot 3 \sqrt{n} > 1$, we can replace $T$ in \eqref{set: T for p lower} by the upper bound $3 \times 45 \log \left(3+\sigma_1  \right) \cdot  \max \big\{ \frac{\sqrt{n}} {K},  \sqrt{r_{p,K}\cdot \log n}  \cdot n^{1/4}  \big\}$.  
Theorem \ref{theo: leadinglower} implies the following corollary.

\begin{corollary} \label{cor: lambdaPWigner} Let $E$ be a sub-Gaussian matrix. Assume that for some parameter $K > 0$,  
\begin{equation} \label{ass: leadinglower}
    \delta_p \geq  270 \log \left(3+\sigma_1  \right) \cdot \max \bigg\{\frac{\sqrt{n}} {K},  \sqrt {r_{p,K} \cdot \log n } \cdot n^{1/4} \bigg\}.
\end{equation}
Then, with probability $1 -o(1)$,
$$\tilde{\lambda}_p \geq \lambda_p - 135 \log \left(3+\sigma_1  \right) \cdot \max \bigg\{ \frac{\sqrt{n}} {K},  \sqrt {r_{p,K} \cdot \log n } \cdot n^{1/4} \bigg\}.$$
    \end{corollary}
This corollary contains Theorem \ref{theoremtoy2} as a special case with  \(p=1\).





\subsubsection{Upper bounds} \label{subsec:upperboundLeading} 
Consider $D=(\lambda_p+T, +\infty)$. If $T < \delta_{p-1}$, then $\delta_D:=\min\{T,\delta_{p-1}-T\}$ for $p > 1$, and $\delta_D = T$ for $p =1$. Similar to the previous subsection, for a parameter $K=K(n) > 0$, let $r_{p-1,K}$ be the number of eigenvalues at least $\lambda_{p-1} -\frac{\delta_{p-1}}{2}-K\|E\|.$
Assume that 
\[
\delta_{p-1} \geq 90 \log \left(3+\textstyle\frac{\sigma_1}{\sqrt { r_{p-1,K} \cdot x \|E\|}}   \right) \cdot \max \bigg\{ \frac{\| E \|} {K},  \sqrt {r_{p-1,K} \cdot  x \|E\| } \bigg\},
\]
and set
\[
T:= 45\log \left(3+\textstyle\frac{\sigma_1}{\sqrt { r_{p-1,K} \cdot x \|E\|}}   \right) \cdot \max \bigg\{ \frac{\| E \|} {K},  \sqrt {r_{p-1,K} \cdot  x \|E\| } \bigg\}.
\]
Similar to the discussion for the lower bound, with this choice of $T$, $T \leq \delta_{p-1}/2$. This implies $\delta_D=T$, $r_K \leq r_{p-1,K}$, and hence similar to \eqref{verify condition for lower bound},
\[ \textstyle
\delta_D \geq  45\log \left(3+\frac{\sigma_1}{\delta_D}   \right) \cdot \max \bigg\{ \frac{\| E \|} {K},  \sqrt {r_{p-1,K} \cdot  x \|E\| } \bigg\} \geq 45\log \left(3+\frac{\sigma_1}{\delta_D}   \right) \cdot \max \bigg\{ \frac{\| E \|} {K},  \sqrt {r_{K} \cdot  x \|E\| } \bigg\}.
\]
Thus, analogous to Theorem \ref{theo: leadinglower} and Corollary \ref{cor: lambdaPWigner}, using Theorem \ref{theo: main0}, we obtain the following upper bounds.
\begin{theorem} \label{theo:leadingupper}
      Assume that for some parameter $K > 0$, 
    \[
    \delta_{p-1} \geq 90 \log \left(3+\textstyle\frac{\sigma_1}{\sqrt { r_{p-1,K} \cdot x \|E\|}}   \right) \cdot  \max \bigg \{  \frac{\| E \|} {K}, \sqrt { r_{p-1,K} \cdot x \|E\| }  \bigg\}.
    \]
    Then
    \[ 
    \tilde{\lambda}_p \leq \lambda_p + 45 \log \left(3+\textstyle\frac{\sigma_1}{\sqrt { r_{p-1,K} \cdot x \|E\|}}   \right) \cdot  \max \bigg \{  \frac{\| E \|} {K}, \sqrt { r_{p-1,K} \cdot x \|E\| }  \bigg\}.
    \]
\end{theorem}

\begin{corollary} \label{cor: lambdapWignerupper}
 Let $E$ be a sub-Gaussian matrix. Assume that for some parameter  $ K > 0$, 
$$ \delta_{p-1} \geq  270 \log \left(3+\sigma_1  \right) \cdot \max \bigg\{ \frac{\sqrt{n}} {K},  \sqrt {r_{p-1,K} \cdot \log n } \cdot n^{1/4} \bigg\}.$$
Then, with probability $1 -o(1)$,
$$\tilde{\lambda}_p \leq \lambda_p + 135 \log \left(3+\sigma_1 \right) \cdot \max \bigg\{  \frac{\sqrt{n}} {K},  \sqrt {r_{p-1,K} \cdot \log n } \cdot n^{1/4} \bigg\}.$$
In particular, for $p=1$, with $r$ denoting the number of eigenvalues at least $\lambda_1 - 3K\sqrt{n}$, we have, without any gap condition, that 
$$ \tilde{\lambda}_1 \leq \lambda_1 + 135 \log \left(3+\sigma_1\right) \cdot \max \bigg\{  \frac{\sqrt{n}} {K},  \sqrt {r \log n } \cdot n^{1/4} \bigg\}.  $$
    \end{corollary}

This corollary contains Theorem~\ref{theoremtoy1} as a special case with \(p=1\).

\subsection{Perturbation of the least singular value} \label{section: least}

The idea is to maintain the stability of an {\it empty} region (a region that does not contain any eigenvalue). Consider $D= (-T, T)$ for $T < \sigma_n$. This region
is empty (with respect to $A$). If $D$ remains empty after the perturbation, 
then $\tilde \sigma_n \ge T$.

Let us work out a specific case as an illustration. 
Choose $T := \sigma_n/2$. We have $D=(-\frac{\sigma_n}{2}, \frac{\sigma_n}{2})$, with $\delta_D=\frac{\sigma_n}{2}$ and $C_D=1$. For a chosen parameter $K=K(n)>0$, in this setting, $r_K$ becomes the number of singular values at most $\sigma_n/2 + K \|E\|$.
By Theorem \ref{theo: main0}, $D$ is stable if 
$$  \delta_D \geq 45 \log \left(3+\frac{\sigma_1}{\delta_D}  \right) \cdot  \max \bigg\{  \frac{\| E \|} {K},  \sqrt {r_K \cdot  x \|E\| } \bigg\},$$
or equivalently 
$$ \sigma_n \geq 90 \log \left(3+\frac{2\sigma_1}{\sigma_n}  \right) \cdot  \max \bigg\{  \frac{\| E \|} {K},  \sqrt {r_K \cdot x \|E\| } \bigg\}.$$
We thus obtain: 
\begin{theorem} \label{theo: least1} Assume that for some parameter $K > 0$, 
 \begin{equation} \label{mainassumption} \sigma_n \geq 90 \log \left(3+\frac{2\sigma_1}{\sigma_n}  \right) \cdot  \max \bigg\{ \frac{\| E \|} {K},  \sqrt {r_K \cdot x \|E\| } \bigg\}. \end{equation} 
 Then $\tilde{\sigma}_n \geq \sigma_n/2$.
\end{theorem}
 The constant  $1/2$ in the claim $\tilde{\sigma}_n \geq \sigma_n/2$ can be replaced by any fixed constant $0 < c< 1$, at the cost of changing the constants $90$ and $2$.

Next, rather than fixing \(T=c\sigma_n\) for some \(0<c<1\), we optimize the choice of \(T\). Let \(r\) denote the number of singular values of \(A\) that are at most \(\sigma_n+K\|E\|\). For
\(
D=(-T,T)
\)
 with $T > \sigma_n/2$,
we have \(\delta_D=\sigma_n-T\) and \(C_D=1\). Thus, Theorem~\ref{theo: main0} guarantees that \(D\) is stable whenever
\(
\sigma_n-T>
45
\log\left(3+\frac{\sigma_1}{\sigma_n-T}\right)
\max\left\{
\frac{\|E\|}{K},\sqrt{rx\|E\|}
\right\}.
\)
Equivalently,
\begin{equation}\label{set: T for least}
T<
\sigma_n-
45
\log\left(3+\frac{\sigma_1}{\sigma_n-T}\right)
\max\bigg\{
\frac{\|E\|}{K},\sqrt{rx\|E\|}
\bigg\}.
\end{equation}
Any \(T\) satisfying \eqref{set: T for least} therefore yields
$$
\tilde\sigma_n\geq T.
$$
We now choose \(T\) explicitly. Assume that
$$
\sigma_n>
90\log\big(3+\textstyle\frac{\sigma_1}{\sqrt{rx\|E\|}}\big) \max\bigg\{
\frac{\|E\|}{K},\sqrt{rx\|E\|}
\bigg\},
$$
and set
$$
T:=\sigma_n-45\log\big(3+\textstyle\frac{\sigma_1}{\sqrt{rx\|E\|}}\big) \max\bigg\{
\frac{\|E\|}{K},\sqrt{rx\|E\|}
\bigg\}.
$$
Then,
$
\delta_D=\sigma_n-T
=45\log\big(3+\frac{\sigma_1}{\sqrt{rx\|E\|}}\big) \max\big\{
\frac{\|E\|}{K},\sqrt{rx\|E\|}
\big\} \geq \sqrt{rx\|E\|}.
$
Consequently,
\(
\log\left(3+\frac{\sigma_1}{\delta_D}\right)
\leq \log(3+\frac{\sigma_1}{\sqrt{rx\|E\|}}),
\)
so \eqref{set: T for least} (i.e., the stability condition in Theorem~\ref{theo: main0}) is satisfied. Hence \(D=(-T,T)\) is stable, yielding the following theorem. 
\begin{theorem}\label{theo: least2}
    Let $r$ denote the number of singular values at most $\sigma_n+K\|E\|$. Assume that for some parameter $K>0$, 
    \[\textstyle
  \sigma_n > 90
\log\left(3+\frac{\sigma_1}{\sqrt{rx\|E\|}}\right)
\max\big\{
\frac{\|E\|}{K},\sqrt{rx\|E\|}
\big\}.
    \]
    Then 
    \[ \textstyle
    \tilde{\sigma}_n \geq \sigma_n - 45
\log\left(3+\frac{\sigma_1}{\sqrt{rx\|E\|}}\right)
\max\big\{
\frac{\|E\|}{K},\sqrt{rx\|E\|}
\big\}.
    \]
\end{theorem}
By Lemma~\ref{Wigner property}, we can bound \(x\) and  \(\|E\|\)
by \(\log n\) and \(3\sqrt n\), respectively. Theorem~\ref{theoremtoy3} then follows immediately.

 
    

\subsection {Relations to existing results} \label{sec: comparison}

In this subsection, we discuss relations between our new results and existing results in the literature.

\subsubsection{Low rank matrices with random perturbation}
The case when  $E$ is random and $A$ has low rank has been studied by many authors, and bounds comparable to ours have been obtained in \cite{OVK13, B-GN1, KY1, CDF1, PRS1, bourgade2016fixed, EBW1}. Their methods differ significantly from ours and do not seem to yield the results in this paper. 

As discussed in Subsection~\ref{subsec: lowrank},  when $E$ is a Wigner matrix and $A$ has low rank with eigenvalues comparable to $\|E\|$, one can have very precise information about the leading eigenvalues of $A+E$ \cite{B-GN1, KY1, CDF1, PRS1, bourgade2016fixed}. However, these assumptions are strong, and no such results are known in our setting, where $E$ is only sub-Gaussian and the eigenvalues of $A$ can be much larger than $\|E\|$.

\subsubsection{A relative perturbation bound} \label{subsubsec: relative bound}
A useful relative perturbation bound \footnote{This bound is called a ``Weyl-type relative perturbation bound'' in the recent extensions \cite{DopicoMoroMolera2000, Ma2023}.} is the following.  Assume that $A$ is positive definite; then one has 
\begin{equation} \label{relative2} \frac{| \tilde \lambda_p -\lambda_p | }{ \lambda_p }  \le \| A^{-1/2} E A^{-1/2}  \|,\,\,\text{or equivalently}\,\,  | \tilde \lambda_p -\lambda_p |    \le \lambda_p \| A^{-1/2} E A^{-1/2}  \|. \end{equation}  
%
%
This bound was pointed out to us by a referee, who also observed that it yields a bound on the least singular value comparable to Theorems~\ref{theo: least1} and~\ref{theo: least2}. We present the detailed argument in Appendix~\ref{secion: app comparision relative}. However, \eqref{relative2} becomes less effective for large eigenvalues $\lambda_p$ and does not yield the other results in this paper.

\subsubsection{Least singular value of perturbed matrices }

In recent years, there has been considerable progress on the least singular value of randomly perturbed matrices; see, e.g., \cite{spielman2004smoothed, Ti1, TV1, TV2, JSS1, JSS2, RuV1}. A typical result in this topic shows that, regardless of the structure of $A$, the least singular value of $\tilde A=A+E$ is at least $n^{-C}$ for some constant $C>0$ with high probability.  The main method used here relies on the inverse Littlewood--Offord theory, and is fundamentally different from the method discussed in this paper.

There is also another line of recent works on the least singular value under perturbations and rounding errors; see, e.g.,\cite{Stewart1984, BoutsikasDrineasIpsen2024, Rump2009, DexterBoutsikasMaIpsenDrineas2025}.
These works mainly concern the regime where $\sigma_n$ is small, possibly
zero,  and establish conditions under which
$\tilde{\sigma}_n$ admits a nontrivial  lower bound. 
Our results here are of a different spirit. We assume that \(\sigma_n\) is sufficiently large, while still allowing \(\sigma_n<\|E\|\), and provide conditions under which \(\tilde{\sigma}_n\) remains close to \(\sigma_n\).

\subsubsection {The derivative of an eigenvalue} 
Matrix analysis allows us to compute the derivative of $\lambda_p$ with respect to a small change. 
In particular, if $\lambda_p$ is simple (i.e., $\delta_{(p)}>0$) and $\| E\| \rightarrow 0$, 
one has $ \tilde \lambda_p = \lambda_p +  u_p^\top E u_p + O\big(\frac{\| E \|^2}{\delta_{(p)}} \big)$; see, e.g., \cite{SS1, GreenbaumLiOverton2020}. 
The linear term $u_p^\top E u_p$ is in the same spirit as our directional perturbation parameter $x$ (see Definition \eqref{def: x}). However, the above expansion is useful only when the \(O\big(\frac{\| E \|^2}{\delta_{(p)}} \big)\) remainder is sufficiently small, a condition that is often too restrictive in applications.


%
%

\section {Proof of Theorem \ref{theorem: main}: Stability of projections} \label{sec: proof}

\subsection{A new eigenspace bound} \label{subsec: space to value}
Following the discussion in Subsection~\ref{subsec: main theorem intro}, we focus on bounding \(\|\tilde\Pi_D-\Pi_D\|\). A classical result closely related to this problem is the Davis--Kahan theorem \cite{DKoriginal}. However, this theorem and its recent refinements (e.g., \cite{tran2025davis, YWS1, SS1, KX1}) require \(\delta_D\geq\|E\|\) to yield a nontrivial bound, and therefore do not improve upon the Weyl criterion (Corollary~\ref{cor:stable}).

Our main result (Theorem~\ref{theorem: main}) provides a new perturbation bound that can guarantee
\(
\|\tilde\Pi_D-\Pi_D\|<1
\)
even when \(\delta_D\) is much smaller than \(\|E\|\), provided, in the spirit of Theorem~\ref{theo: main0}, that \(x\) is small and the eigenvalues do not cluster near \(D\). As usual, for a parameter \(K=K(n)>0\), let \(r_K\) denote the number of eigenvalues whose distance from \(D\) is at most \(K\|E\|\).


 %
%
  %
We are going to prove the following  more technical 
(but slightly stronger) version of Theorem \ref{theorem: main}.

\begin{theorem}\label{theorem: main unified bound}
Assume that, for some \(K>0\),
$$
\delta_D\geq
6\max\left\{
\sqrt{r_Kx\|E\|},
\frac{\|E\|}{K}
\right\}.
$$
Then
\begin{equation}\label{main bound uni}
\|\tilde\Pi_D-\Pi_D\|
\leq
8C_D\left[
\frac{r_Kx}{\delta_D}
+
\log\left(
9+\frac{5\sigma_1}{\min\{\delta_D,\|E\|\}}
\right)
\left(
\frac1K
+\frac{\|E\|}{K\delta_D}
+\frac{r_K\cdot x\|E\|}{\delta_D^2}
\right)
\right].
\end{equation}
\end{theorem}
Theorem~\ref{theorem: main} follows from Theorem~\ref{theorem: main unified bound} by simplifying the right-hand side of \eqref{main bound uni} separately in the two cases
\(\delta_D\geq\|E\|\) and \(\delta_D<\|E\|\),  together with the elementary bound $16 \log\left(
9+\frac{5\sigma_1}{\min\{\delta_D,\|E\|\}}
\right) < 44 \log \left(3+\frac{\sigma_1}{\min\{\delta_D,\|E\|\}} \right).$


Before proving Theorem~\ref{theorem: main unified bound}, we record three standard observations, which will help us to make certain assumptions in the proof. 

The first is a corollary of the Davis--Kahan theorem, adapted to our stability setting.\footnote{For the original version, see Theorem~VII.3.2 in \cite{Book1}.}

\begin{theorem}[Davis--Kahan, stable version]\label{DK}
Let \(D\) be a region. Then
\(
\|\tilde\Pi_D-\Pi_D\|
\leq
\frac{\pi\|E\|}{2\delta_D}.
\)
\end{theorem}
By Theorem~\ref{DK}, if
\(
r_K \cdot x\geq\|E\|,
\)
then \eqref{main bound uni} follows immediately by just using 
the first term on the right-hand side. Thus, it suffices to consider the case 
\[
r_K \cdot x<\|E\|.
\]
Since \(u_1,\ldots,u_n\) form an orthonormal basis, for every \(v\in\mathbb R^n\) with \(\|v\|=1\), we may write
\(
v=\sum_{i=1}^n\alpha_i u_i,
\)
where
\(
\sum_{i=1}^n \alpha_i^2=1.
\)
Hence,
\[
 |v^\top E v| = |\sum_{1\leq i,j \leq n} \alpha_i \alpha_j (u_i^\top E u_j)| \leq x \cdot (\sum_{1 \leq i \leq n} |\alpha_i|)^2 \leq n x.
\]
Therefore,
\(
\|E\|
=
\max_{\|v\|=1}|v^\top Ev|
\leq nx.
\)
It follows from \(r_Kx<\|E\|\) that \(r_K<n\). Thus, the \(K\|E\|\)-neighborhood of \(D\) cannot contain the entire spectrum of \(A\). Using the definition of $\delta_D$, we can deduce that 
\begin{equation}\label{additional assum: KE<2sigma+delta}
K\|E\|\leq2\sigma_1+\delta_D.
\end{equation}

Second, by Weyl's inequality, neither \(A\) nor \(\tilde A\) has eigenvalues outside
\(
[\lambda_n-\|E\|,\lambda_1+\|E\|].
\)
Thus, only the portion of \(D\) intersecting this interval contributes to the corresponding spectral projections.
Thus, by truncation if necessary, we can assume that $D$ is contained in this interval. (The truncation only reduces $\delta_D$.) This reduction  allows us to assume
\begin{equation}\label{additional ass: detal<E+sigma}
\delta_D\leq\max\{\|E\|,\sigma_1\}.
\end{equation}

Finally, for any two orthonormal projections $\tilde{\Pi}_D$ and $\Pi_D$,
\(
\|\tilde\Pi_D-\Pi_D\|\leq 1.
\)
Hence, \eqref{main bound uni} is immediate whenever \( \frac{5C_D}{K}
\log\left(
9+\frac{5\sigma_1}{\min\{\delta_D,\|E\|\}}
\right)
\geq 1.\)
It therefore suffices to consider
\begin{equation}\label{additional assum: K>5} \textstyle
K
\geq
5 C_D
\log\left(
9+\frac{5\sigma_1}{\min\{\delta_D,\|E\|\}}
\right)
>
5\log9
> 10. \end{equation}

\subsection {A contour integral representation of the perturbation} \label{subsec: bound eigen}
We now outline the main ideas behind the proof of Theorem~\ref{theorem: main unified bound}. By the discussion in the previous subsection, we can assume that  \eqref{additional assum: KE<2sigma+delta}, \eqref{additional ass: detal<E+sigma}, and \eqref{additional assum: K>5} hold for the remainder of the proof. Shifting $D$ slightly if necessary, we can also assume that the endpoints of $D$ are not eigenvalues of $\tilde{A}$.

In the first step, we provide a contour integral representation of the perturbation $ \tilde\Pi_{D} - \Pi_D$. 
We start with the Cauchy integral formula~\cite{CAbook}. 
 \begin{theorem}[Cauchy integral formula] \label{theo: Cauchy}
Let $\Gamma$ be a simple closed contour. Then
\begin{equation}\label{Cauchy0} 
 \frac{1}{2 \pi {\bf i}} \int_{\Gamma} \frac{1}{z-a}\,dz 
   = 
   \begin{cases}
      1, & a \text{ inside } \Gamma, \\[4pt]
      0, & a \text{ outside } \Gamma.
   \end{cases}  
\end{equation}
 \end{theorem}

We construct $\Gamma$ as follows. 
For each connected component of $D$ with endpoints $x_0<x_1$, we form a rectangle with vertices
\[
(x_0,H),\ (x_0,-H),\ (x_1,-H),\ (x_1,H),
\qquad H=2(\sigma_1+\|E\|+\delta_D).
\]
We define $\Gamma$ as the union of these rectangles. 
By construction, $\Gamma$ encloses exactly the eigenvalues $\lambda_i \in \Lambda_D(A)$; all $\lambda_j \notin \Lambda_D(A)$ lie outside $\Gamma$.
%
%
 %
 Thus, applying the Cauchy integral formula, one obtains the following contour formula (often referred to as the Riesz projection formula), frequently used in numerical analysis  \cite{Book1, Kato1}:
 \begin{equation} \label{contour-formula} 
 \Pi_D:= \sum_{ \lambda_i \in \Lambda_D(A)} u_i u_i ^\top =  \frac{1} {2 \pi {\bf i }} \int_{\Gamma}  (z-A)^{ -1} dz.  \end{equation} 
Similarly, by the definition of $\tilde \Pi_{D}$,
we also have $ \textstyle  \frac{1}{2 \pi {\bf i}}  \int_{\Gamma}   (z-\tilde A)^{-1} dz  = \sum_{ \tilde\lambda_i \in \Lambda_D(\tilde{A})} \tilde u_i  \tilde u_i^\top := \tilde\Pi_{D}.   $
 %
 %
Thus, we obtain a contour integral representation for the perturbation 
 \begin{equation} 
   \tilde\Pi_{D} - \Pi_D=  \frac{1} {2 \pi {\bf i} } \int_{\Gamma}  [(z-\tilde A)^{-1}- (z- A)^{-1} ]  dz. 
  \end{equation}

 \subsection {Bounding the integral by a series} 
We make use of the resolvent formula 
$$(z-\tilde A)^{-1}-(z- A)^{-1}   = (z-A)^{-1} E (z- \tilde A)^{-1}.$$
 Using  this formula repeatedly, we obtain (formally at least) 
\begin{equation} \label{TaylorEx}\textstyle  (z-\tilde A)^{-1}-(z- A)^{-1}  =\sum_{s=1}^{\infty}  (z-A)^{-1} [ E(z-A)^{-1} ]^s .
\end{equation} 
(To make this identity valid, we need to make sure that the right-hand side converges, but let us skip this issue for a moment.) 
Now set 
\begin{equation*}
\textstyle  F_s := \frac{1}{2\pi \textbf{i}}\int_{\Gamma} (z-A)^{-1} [ E (z-A)^{-1} ]^s dz.
\end{equation*}
We have 
\begin{equation} \label{Taylor expansion} \textstyle
    \tilde\Pi_{D} - \Pi_D=   \sum_{s=1}^{\infty} F_s.
\end{equation}
Therefore, by the triangle inequality, 
\begin{equation} \label{bound0} \textstyle
\| \tilde\Pi_{D} - \Pi_D\| \le  \sum_{s=1}^{\infty} \|F_s\|.
\end{equation}


Several previous works have used the Taylor expansion idea; see, for instance, \cite[Chapter 2]{Kato1}. 
The key matter is how to bound $\| F_s \|$. In \cite[Chapter 2]{Kato1}, one simply uses 
\begin{equation}\label{trivialF_s}
   \textstyle  \| F_s \|  \le  \frac{\|E\|^{s}}{2\pi} \int_{\Gamma} \|(z-A)^{-1}\|^{s+1} |dz| = \tilde{O} \left[  \left( \frac{\|E\|}{\delta_D} \right)^s \right]. 
\end{equation}
Summing over $s$, we obtain the classical Davis--Kahan bound, up to a constant factor. 
 Recent works (e.g., \cite{KL1, KX1, JW1, JW2}) have refined this approach by computing the linear term $\|F_1\|$  exactly,  and bounding the other $\|F_s\|,\, s \geq 2$ separately.  For instance, in \cite{KL1}, using \eqref{trivialF_s} again, one has
$$\textstyle  \sum_{s=2}^{\infty} \|F_s\| = O \left[ \left(\frac{\|E\|}{\delta_D} \right)^2   \right].$$
This way, one obtains 
$$\textstyle  \| \tilde\Pi_{D} - \Pi_D\| \leq \|F_1\| + O \left[  \left( \frac{\|E\|}{\delta_D} \right)^2  \right].$$ 

In \cite{tran2025davis}, the authors used a bootstrapping argument to bound the series. It is shown  that if $\delta_D > 2\|E\|$, then
$$ \| \tilde\Pi_{D} - \Pi_D\| \leq C_D \cdot O(F_1'),\,\text{for}\,\,F_1':= \frac{1}{2\pi} \int_{\Gamma} \| (z-A)^{-1} E (z-A)^{-1}\| |dz|.$$
We can show that the right-hand side is 
$  \tilde{O}\left( \frac{r_K\cdot x}{\delta_D} + \frac{1}{K}\right), $ which  is sharp up to a polylog factor. 

The main obstacle here is that  all of these
 refinements require the gap condition $\delta_D > \|E\|$ to start with, and this stops us from getting an improvement over the Weyl criterion in Corollary \ref{cor:stable}.

\subsection{A new way to bound the contour integral: the double-jump strategy} \label{section: boundintegral}
Here we use the double-jump strategy described in Subsection \ref{subsec: main theorem intro} to handle 
\[
F_s=\frac{1}{2\pi \textbf{i}} \int_{\Gamma} G_s (z)dz,\, \text{where}\,\,G_s (z) := (z-A)^{-1} [E(z-A)^{-1}]^s.
\]
 However, this strategy needs to be carried out with some care. In particular, we need to design the contour properly, and the length of the vertical segments will play an important role in the analysis.

  We construct $\Gamma$ as a union of rectangles whose vertical edges are associated with the boundary of $D$.  For each $x_0 \in \partial D$, the corresponding vertical edge connects $(x_0,H)$ and $(x_0,-H)$. Here  $H$ is the height of $\Gamma$ relative to the real axis. We set $H = 2(\sigma_1+\|E\|+\delta_D)$.


The critical step is to bound  the double-jump ratio 
$$G_{s+2} (z) /G_s (z) =E(z-A)^{-1} E (z-A)^{-1}. $$

\begin{lemma}\label{lem: double-jump} Consider the contour $\Gamma$ constructed  above. Under the assumption of Theorem \ref{theorem: main unified bound}, we have,  for all $z \in \Gamma$,
    \begin{equation} \label{bounding E(z-A)E(z-A)l}
 \| E(z-A)^{-1} E (z-A)^{-1}\| \leq h :=    \frac{r_{K}\cdot x \|E\|}{\delta_D^2} + \frac{2\|E\|}{\delta_D K} + \frac{1}{K^2}<\frac{1}{2}. 
\end{equation}
\end{lemma}

Notice that \eqref{bounding E(z-A)E(z-A)l} also guarantees that the Taylor expansion in \eqref{TaylorEx} converges. As a corollary of this lemma, we can control the terms in the series as follows: 
 \begin{itemize}
     \item For each even natural number $s$, 
    \begin{equation*}
        \begin{split}
        \|F_s\| &\leq \frac{1}{2\pi} \int_{\Gamma}\| (z-A)^{-1} [E(z-A)^{-1}]^{s}\| |dz|  \\ &\leq \frac{1}{2\pi}\max_{z \in \Gamma} \|[E(z-A)^{-1}]^{s}\| \cdot \int_{\Gamma}\|(z-A)^{-1}\|  |dz| \\
    &\le \frac{h^{s/2}}{2\pi} \cdot \int_{\Gamma}\|(z-A)^{-1}\|\,\,\,  |dz|.    
        \end{split}
    \end{equation*}

    \item For each odd natural number $s$, 
    \begin{equation*}
        \begin{split}
             \| F_s\| &\le  \frac{1}{2\pi} \int_{\Gamma}\| (z-A)^{-1} [E(z-A)^{-1}]^{s}\| |dz| \\ &\leq  \frac{\max_{z \in \Gamma} \|[E(z-A)^{-1}]^{s-1}\|}{2\pi} \cdot \int_{\Gamma}\|(z-A)^{-1} E (z-A)^{-1}\|  |dz|\\
   &\le \frac{h^{(s-1)/2}}{2\pi} \cdot \int_{\Gamma}\|(z-A)^{-1} E (z-A)^{-1}\|\,\,\,  |dz|.
        \end{split}
    \end{equation*}

 \end{itemize}   
\noindent Consequently, 
\begin{equation}
    \begin{split}
   \sum_{s=1}^\infty \|F_s\| & = \textstyle  \sum_{k=1}^\infty \|F_{2k}\| + \sum_{k=0}^\infty \|F_{2k+1}\|    \\
   & \leq \textstyle  \left( \frac{\int_{\Gamma}\|(z-A)^{-1}\|\,\,\,  |dz|}{2\pi} \right) \cdot \left(\sum_{k=1}^\infty h^k \right) + \left(\frac{\int_{\Gamma}\|(z-A)^{-1} E (z-A)^{-1}\|\,\,\,  |dz|}{2\pi} \right) \left(\sum_{k=0}^\infty h^k \right) \\
   & = \textstyle \frac{1}{(1-h)} \left(h \cdot \frac{\int_{\Gamma}\|(z-A)^{-1}\|\,\,\,  |dz|}{2\pi} + \frac{\int_{\Gamma}\|(z-A)^{-1} E (z-A)^{-1}\|\,\,\,  |dz|}{2\pi}  \right). 
    \end{split}
\end{equation}

It remains to bound the two integrals on the right-hand side. We are going to show 
    \begin{equation} \label{Est: AEA}
     \frac{\int_{\Gamma}\|(z-A)^{-1} E (z-A)^{-1}\|\,\,\,  |dz|}{2\pi C_D} \leq  \frac{r_{K}\cdot x}{\delta_D} +  \frac{1}{K} \cdot \left[1+\frac{1}{\pi}+\frac{9}{\pi}\log \left( 4+\frac{4 \sigma_1}{\min\{\delta_D,\|E\|\}} \right) \right] .   
    \end{equation}
    \begin{equation} \label{Est: zAinverse}
     \frac{\int_{\Gamma}\|(z-A)^{-1}\|\,\,\,  |dz|}{2\pi C_D} \leq \frac{2+ 4 \log \left(9+\frac{5\sigma_1}{\min\{\delta_D,\|E\|\}} \right)}{2 \pi} .       \end{equation}

Combining these bounds and the fact $h <\frac{1}{2}$ (Lemma \ref{lem: double-jump}), we obtain 
\begin{equation} \label{sumFsbound0}
    \frac{\sum_{s=1}^\infty \|F_s\|}{C_D} \leq  \frac{2+ 4 \log \left(9+\frac{5\sigma_1}{\min\{\delta_D,\|E\|\}} \right)}{2 \pi}\cdot h+ \frac{2r_{K}\cdot x}{\delta_D} +  \frac{2}{K} \cdot \left[1+\frac{1}{\pi}+\frac{9}{\pi}\log \left( 4+\frac{4 \sigma_1}{\min\{\delta_D,\|E\|\}} \right) \right].
\end{equation}

\vskip2mm The rest of the proof is a simple bookkeeping. 
The constants in the theorem  have been set up, generously and with foresight, to carry out this task. 

Note that together, \eqref{additional assum: KE<2sigma+delta},\eqref{additional ass: detal<E+sigma}, and \eqref{additional assum: K>5} imply that 
\(
\frac{\sigma_1}{\min\{\delta_D,\|E\|\}} \geq 3; 
\) see Lemma \ref{lem: tech0}.
Therefore, 
\(
\log \left(9+\frac{5\sigma_1}{\min\{\delta_D,\|E\|\}} \right) > 3.
\)
Hence \eqref{sumFsbound0} yields
\begin{equation*} 
    \begin{split}
   &\frac{\sum_{s=1}^\infty \|F_s\|}{C_D} \leq \frac{7\log \left(9+\frac{5\sigma_1}{\min\{\delta_D,\|E\|\}} \right)}{3\pi} \cdot h+ \frac{2r_{K}\cdot x}{\delta_D} + \frac{22\log \left(9+\frac{5\sigma_1}{\min\{\delta_D,\|E\|\}} \right)}{ \pi K}.
    \end{split}
\end{equation*}
Substituting $h=  \frac{r_{K}\cdot x \|E\|}{\delta_D^2} + \frac{2\|E\|}{\delta_D K} + \frac{1}{K^2}$ and multiplying both sides by $C_D$, we obtain
\begin{equation}\label{sumFsbound1}
\begin{split}
 \sum_{s=1}^\infty \|F_s\| & \textstyle \leq C_D \left[\frac{2 r_K \cdot x}{\delta_{D}} + \frac{22\log \left(9+\frac{5\sigma_1}{\min\{\delta_D,\|E\|\}} \right) }{\pi K}+ \frac{7 \log \left(9+\frac{5\sigma_1}{\min\{\delta_D,\|E\|\}} \right)}{ 3\pi} \cdot \left( \frac{2\|E\|}{K \delta_D} + \frac{r_K\cdot x \|E\|}{\delta_D^2}+\frac{1}{K^2} \right)\right] \\
 & \textstyle \leq C_D \left[\frac{2 r_K \cdot x}{\delta_{D}} + \frac{22\log \left(9+\frac{5\sigma_1}{\min\{\delta_D,\|E\|\}} \right)}{\pi K}+ \frac{7 \log \left(9+\frac{5\sigma_1}{\min\{\delta_D,\|E\|\}} \right)}{ 3\pi} \cdot \left( \frac{2\|E\|}{K \delta_D} + \frac{r_K\cdot x \|E\|}{\delta_D^2}+\frac{1}{10K} \right)\right] \\
 & \textstyle \leq 8 C_D \left[\frac{ r_K \cdot x}{\delta_{D}} + \log \left(9+\frac{5\sigma_1}{\min\{\delta_D,\|E\|\}} \right) \cdot \left( \frac{1}{K}+ \frac{\|E\|}{K \delta_D} + \frac{r_K\cdot x \|E\|}{\delta_D^2} \right)\right].
\end{split}
\end{equation}
The second and third inequalities follow from the facts that $K \geq 10$ and $$\textstyle \log \left(9+\frac{5\sigma_1}{\min\{\delta_D,\|E\|\}} \right)> 3.$$
Finally, since $\|\tilde \Pi_D -\Pi_D\| \leq \sum_{s=1}^\infty \|F_s\|$, \eqref{sumFsbound1} proves Theorem \ref{theorem: main unified bound}.


We prove Lemma~\ref{lem: double-jump} in the next subsection. In this proof, we  make essential use of the directional perturbation \(x\) and the density information of the spectrum. The estimates \eqref{Est: AEA} and \eqref{Est: zAinverse} follow the strategy developed in our previous work \cite{tran2025davis}; we provide the details in Appendix~\ref{section: tech estimates}.

\subsection{Proof of Lemma \ref{lem: double-jump}} \label{subsec: proveboundH} Our task is to show
\[
\| E(z-A)^{-1} E (z-A)^{-1}\| \leq    \frac{r_{K}\cdot x \|E\|}{\delta_D^2} + \frac{2\|E\|}{\delta_D K} + \frac{1}{K^2}.
\]
To start, we  decompose each resolvent factor as
$$
(z-A)^{-1}=P+Q,\,\,\text{where}
$$
\begin{equation}\label{def: PQ}
P:=\sum_{k\in I_K(D)}\frac{u_ku_k^\top}{z-\lambda_k},
\qquad
Q:=\sum_{l\notin I_K(D)}\frac{u_lu_l^\top}{z-\lambda_l}.
\end{equation}
Define
\(
I_K(D):=\{i\in[n]:\operatorname{dist}(\lambda_i,D)\leq K\|E\|\}.
\)
This is the index set of eigenvalues near $D$. By definition, we have 
\(
|I_K(D)|=r_K.
\)

Next, we use this decomposition to split $\| E(z-A)^{-1} E (z-A)^{-1}\|$ into terms involving \(E,P,\) and $Q$, which we then estimate separately. The estimates will exploit the explicit construction of \(\Gamma\) in Step~1 of Subsection~\ref{section: boundintegral}, together with the definition of $x:=\max_{1\leq i,j\leq n}|u_i^\top E u_j|.$

Splitting $(z-A)^{-1} = P+Q$, we obtain 
$$\textstyle E(z-A)^{-1} E (z-A)^{-1} = EPEP+EPEQ+EQEP+EQEQ = M_1+ M_2+ M_3+ M_4,$$
where 
\begin{equation}
    \begin{split}
       &\textstyle M_1 := EPEP= \sum_{ i,j \in I_K(D)} E \frac{u_i u_i^\top}{z -\lambda_i} E \frac{u_j u_j^\top}{z -\lambda_j}, \\
       &\textstyle M_2 : = EPEQ = \sum_{i \in I_K(D), j \notin I_K(D)} E \frac{u_i u_i^\top}{z -\lambda_i} E \frac{u_j u_j^\top}{z -\lambda_j}, \\
       &\textstyle M_3 : = EQEP= \sum_{i \notin I_K(D), j \in I_K(D)} E \frac{u_i u_i^\top}{z -\lambda_i} E \frac{u_j u_j^\top}{z -\lambda_j}, \\
       &\textstyle M_4 := EQEQ =  \sum_{ i,j \notin I_K(D)} E \frac{u_i u_i^\top}{z -\lambda_i} E \frac{u_j u_j^\top}{z -\lambda_j}. 
    \end{split}
\end{equation}
By the triangle inequality, we have 
$ \textstyle \|E(z-A)^{-1} E (z-A)^{-1}\| \leq \sum_{i=1}^4 \|M_i\|.$

{\noindent \it Bounding $\|M_1\|$.} Factoring the left factor $E$, we have 
\[
\|M_1\|=\|EPEP\| \leq \|E\| \cdot \|PEP\|. 
\]
By the definition of the spectral norm, we further have 
$$\|M_1\| \leq \|E\| \cdot \max_{\|\textbf{
v
}\|=\|\textbf{w}\|=1} \textbf{v}^\top  \sum_{ i,j \in I_K(D)}  \frac{u_i u_i^\top}{z -\lambda_i} E \frac{u_j u_j^\top}{z -\lambda_j} \textbf{w}.$$
By the triangle inequality, the right-hand side is at most 
$$\|E\| \cdot \max_{\|\textbf{
v
}\|=\|\textbf{w}\|=1} \sum_{ i,j \in I_K(D)} |\textbf{v}^\top   u_i| \cdot \bigg| \frac{ u_i^\top E u_j}{(z -\lambda_i)(z-\lambda_j)}  \bigg|\cdot  |u_j^\top \textbf{w}|.$$
By the definition of $x$, we have
\(
|u_i^\top E u_j| \leq x
\)
for all  $i,j \in I_K(D)$. Therefore, 
\begin{equation} \label{M1 to xE}
    \begin{split}
\|M_1\|  & \textstyle \leq \|E\| \cdot \max_{\|\textbf{
v
}\|= \|\textbf{w}\|=1} \sum_{ i,j \in I_K(D)} |\textbf{v}^\top   u_i| \cdot \bigg| \frac{ u_i^\top E u_j}{(z -\lambda_i)(z-\lambda_j)}  \bigg|\cdot  |u_j^\top \textbf{w}| \\
&\leq\textstyle  x \|E\| \cdot         \max_{\|\textbf{
v
}\|=\|\textbf{w}\|=1} \sum_{ i,j \in I_K(D)} |\textbf{v}^\top   u_i| \cdot  \frac{ 1}{|(z -\lambda_i)(z-\lambda_j)|}  \cdot  |u_j^\top \textbf{w}| \\
&\textstyle  = x \|E\| \cdot  \max_{\|\textbf{
v
}\|=\|\textbf{w}\|=1} \bigg( \sum_{i \in I_K(D)} \frac{|\textbf{v}^\top   u_i| }{|z -\lambda_i|} \bigg) \cdot \bigg( \sum_{j \in I_K(D)} \frac{|u_j^\top \textbf{w}|}{|z- \lambda_j|} \bigg).
    \end{split}
\end{equation}
Furthermore, applying the Cauchy--Schwarz inequality on $ \left( \sum_{i \in I_K(D)} \frac{|\textbf{v}^\top  E u_i| }{|z -\lambda_i|} \right)$, we obtain
\[ \textstyle
\left( \sum_{i \in I_K(D)} \frac{|\textbf{v}^\top   u_i| }{|z -\lambda_i|} \right) \leq \sqrt{ \left( \sum_{i \in I_K(D)} |\textbf{v}^\top u_i| ^2\right) \left( \sum_{i \in I_K(D)} \frac{1}{|z- \lambda_i|^2} \right)} \leq \sqrt{\sum_{i \in I_K(D)} \frac{1}{|z- \lambda_i|^2} }.
\]
The last inequality follows from the fact that $\|\textbf{v}\|=1$ and $\{u_1, u_2, \cdots, u_n\}$ is an orthonormal system. Similarly, 
\(
\textstyle
\left( \sum_{j \in I_K(D)} \frac{|u_j^\top \textbf{w}|}{|z- \lambda_j|} \right) \leq  \sqrt{\sum_{j \in I_K(D)} \frac{1}{|z- \lambda_j|^2} }
\) for any $\|\textbf{w}\|=1$, and hence
\begin{equation} \label{CS1 M1} \textstyle
    \max_{\|\textbf{
v
}\|=\|\textbf{w}\|=1} \left( \sum_{i \in I_K(D)} \frac{|\textbf{v}^\top   u_i| }{|z -\lambda_i|} \right) \cdot \left( \sum_{j \in I_K(D)} \frac{|u_j^\top \textbf{w}|}{|z- \lambda_j|} \right) \leq \sum_{j \in I_K(D)} \frac{1}{|z-\lambda_j|^2}.
\end{equation}
By the construction $\Gamma$ and the definition of $\delta_D:=\min_{i \in [n]} \mathrm{dist}(\lambda_i, \partial D)$, we have that for all $z \in \Gamma$,
\begin{equation} \label{deltaD M1}
   \sum_{j \in I_K(D)} \frac{1}{|z-\lambda_j|^2} \leq \max_{a\in \partial D} 
\sum_{j\in I_K(D)}
\frac{1}{|a-\lambda_j|^2} \leq \sum_{j\in I_K(D)}\frac{1}{\delta_D^2} =  \frac{r_K}{\delta_D^2}\,\,(\text{since $r_K=|I_K(D)|$}).
\end{equation}
Together, \eqref{M1 to xE}, \eqref{CS1 M1}, and \eqref{deltaD M1} imply
$$  \|M_1\| \leq x\|E\| \cdot \frac{r_K}{\delta_D^2} =  \frac{r_{K}\cdot x \|E\|}{\delta_D^2}.$$

{\noindent \it Bounding $\|M_2\|, \|M_3\|, \|M_4\|$.} We first have
\begin{equation} \label{M2}
    \begin{split}
\|M_2\|=\|EPEQ\| & = \textstyle \bigg\|E \left(\sum_{i \in I_K(D)} \frac{u_i u_i^\top}{z -\lambda_i} \right) E \left( \sum_{j \notin I_K(D)} \frac{u_j u_j^\top}{z -\lambda_j} \right) \bigg\| \\
& \leq \textstyle\|E\| \cdot \bigg\| \sum_{i \in I_K(D)} \frac{u_i u_i^\top}{z -\lambda_i} \bigg\| \cdot \|E\| \cdot \bigg\|\sum_{j \notin I_K(D)} \frac{u_j u_j^\top}{z -\lambda_j}\bigg\| \\
& \leq \textstyle \|E\|^2 \times \frac{1}{\min_{z \in \Gamma, i \in I_K(D)}|z -\lambda_i|} \times \frac{1}{\min_{z \in \Gamma, j \notin I_K(D)}|z -\lambda_j|} \\
& \leq \frac{\|E\|^2}{\delta_D K\|E\|}= \frac{\|E\|}{\delta_D K}.
    \end{split}
\end{equation}
The last inequality follows from the construction of $\Gamma$.

The same argument gives 
\begin{equation} \label{M3,M4}
 \|M_3\|=\|EQEP\|    \leq \frac{\|E\|^2}{\delta_D K\|E\|}=\frac{\|E\|}{\delta_D K} \,\text{and}\,\, \|M_4\|=\|EQEQ\| \leq \frac{\|E\|^2}{(K\|E\|)^2} = \frac{1}{K^2}. 
\end{equation}

Combining the estimates for $\|M_1\|, \|M_2\|, \|M_3\|, \|M_4\|$, we obtain 
$$ \textstyle\|E(z-A)^{-1} E (z-A)^{-1}\| \leq \sum_{i=1}^4 \|M_i\| \leq  \frac{r_{K}\cdot x \|E\|}{\delta_D^2} + \frac{2\|E\|}{\delta_D K} + \frac{1}{K^2}. $$
This proves Lemma \ref{lem: double-jump}.
\begin{remark} [Longer jumps]
In view of the double-jump strategy, one may consider longer jumps (a triple-jump, say). However, Steps 2 and 3 become substantially more technical, with no significant gain. 
\end{remark}

\section{Concluding remarks} \label{sec: remark and extension}

\subsection{A possible refinement given more spectral information}  Define 
\[
S_{K}(D):= \max_{a\in \partial D} 
\sum_{j\in I_K(D)}
\frac{1}{|a-\lambda_j|^2}.
\]
In the proof of Theorem~\ref{theorem: main} (see \eqref{deltaD M1}), we use the trivial estimate that 
\[
S_{K}(D) \leq \frac{r_K}{\delta_D^2},
\]
which holds since $|a -\lambda_j| \geq \delta_D$ for all  $a \in \partial D, j \in  I_K(D)$. 
%

We can state all of our results 
using this quantity $S_K(D)$ directly, skipping the above estimate. 
When additional information about the distribution of the eigenvalues of $A$ is available, we can have a better estimate for $S_K(D)$  than $\frac{r_K}{\delta_D^2}$, leading to better perturbation bounds. 
In particular,  we can replace Theorem~\ref{theorem: main} by the following
\begin{theorem}
     Assume that for some $ K >0$,
$$ \delta_D \geq  \frac{6\|E\|}{K}\,\,\,\text{and}\,\,S_K(D) \leq \frac{1}{36x \|E\|}. $$ 
\begin{itemize}
    \item If $\delta_D \geq \|E\|$, then 
    \begin{equation}
   \| \tilde\Pi_{D} - \Pi_D\| \leq 44 \log \left(3+\frac{\sigma_1}{\|E\|}  \right) \cdot  C_D \cdot \left(x \sqrt{r_K\cdot S_K(D)}  +   \frac{1}{K} \right).      
    \end{equation}
    \item If $\delta_D < \|E\|$,  then
    \begin{equation} 
         \| \tilde\Pi_{D} - \Pi_D\| \leq 44 \log \left(3+\frac{\sigma_1}{\delta_D}  \right) \cdot  C_D \cdot \left( \frac{\|E\|}{K \delta_D} + x\|E\| \cdot S_K(D) \right).
    \end{equation}
\end{itemize}

\end{theorem}



\subsection{Potential applications for the double-jump argument} Our double-jump argument (see Subsection~\ref{subsec: main theorem intro}) is robust and can be applied to general matrix series. As discussed in Subsection~\ref{subsec: main theorem intro}, its strength comes from the fact that when a matrix \(M\) is not symmetric (or, more generally, not normal), \(\|M^2\|\) can be much smaller than \(\|M\|^2\). Thus, when considering a series of the form
$
\sum_{i=1}^{\infty} M^i,
$
it can be more effective to exploit double jumps, based on \(M^2\), rather than single jumps based on \(M\).

In particular, in the context of matrix perturbation, one can consider the following general setting. Let
$
A=U\Lambda U^\top
$
be the spectral decomposition of \(A\), where \(\Lambda\) is diagonal with diagonal entries given by the eigenvalues of \(A\). Given a function \(f\), define
$
f(A):=Uf(\Lambda)U^\top,
$
where \(f(\Lambda)\) is obtained by applying \(f\) to each diagonal entry of \(\Lambda\). For instance, if \(f(z)=z^2\), then \(f(A)=A^2\). The projection of $f(A)$ onto a subspace defined by a region $D$ is 
$$\Pi_{D,f}  = U f_D(\Lambda) U^T, $$ 
where $f_D(\Lambda)$ is obtained from $f(\Lambda)$ by zeroing out all entries $f(\lambda_i)$ if $\lambda_i$ is not in $D$. The case $f =1$ is considered in this paper. The case $f(z)=z$ corresponds to the truncated spectral decomposition of $A$ using only the eigenvalues in $D$. In particular, if $D = (T, \infty)$, then $\Pi_{D,f}$ is the low-rank approximation of $A$ using only eigenvalues larger than a threshold $T$.

Similarly to \eqref{id: Gs Pi}, we have 
\begin{equation}  \label{id: Gs Pi1}
    \tilde\Pi_{D,f} -\Pi_{D,f} = \frac{1}{2\pi \textbf{i}} \int_{\Gamma } \sum_{s=1}^{\infty} f(z)  G_s(z) \, dz, \end{equation} where $ G_s(z):= (z-A)^{-1} [ E(z-A)^{-1}]^{s}$. It is clear that we can apply 
the double-jump argument here in exactly the same way.

\subsection{Refining the parameter $x$} Consider the critical parameter
\[
x:= \max_{u, v \in V} | u^\top E v |,
\]
where $V$ is an orthonormal eigenbasis of $A$. 

A closer look at the proofs reveals that it suffices to consider only the \emph{relevant} eigenvectors. Indeed,  \(x\) enters the proof only in the estimate of \(M_1:= EPEP\), where the indices of the eigenvectors appearing in $P$ all belong to \(I_K(D)\). Thus, our results remain valid with $x$ replaced by
\[
x_K := \max_{u,v \in V_K} |u^\top E v|,
\]
where $V_K$ denotes the set of eigenvectors whose corresponding eigenvalues lie within a $K\|E\|$-neighborhood of $D$.


In the random setting, \(x\) is already small (of order \(o(\log n)\); see Lemma \ref{Wigner property}), so this refinement does not substantially improve our general probabilistic bounds. In the deterministic setting, however, \(x_K\) can be significantly smaller than \(x\).


\subsection{Extension to Hermitian matrices}
Our results extend naturally to Hermitian matrices \(A\) and \(E\), since the eigenvalues of \(A\) and \(\tilde A\) remain real. The stability theorems and perturbation bounds remain unchanged, except that the directional perturbation \(x\) is now defined using the complex eigenvectors. Specifically, given the spectral decomposition
\(
A=\sum_{i=1}^n \lambda_i u_i u_i^*,
\)
where \(V=\{u_1,\ldots,u_n\}\) is an orthonormal eigenbasis of \(A\) in \(\mathbb{C}^n\), we define
$$
x:=\max_{u,v\in V}|u^*Ev|.
$$
\subsection{Extension to rectangular matrices} \label{sec: symtorec}
We can extend our study to rectangular matrices, with respect to the stability of the singular values, using the standard symmetrization argument. Details will appear elsewhere.

\bibliography{RefO}
\appendix
\section{Proofs of the technical estimates} \label{section: tech estimates}
In this section, we prove the estimates \eqref{Est: AEA} and~\eqref{Est: zAinverse}.
Recall that it is sufficient to prove \eqref{Est: AEA} and~\eqref{Est: zAinverse} in the setting that 
\begin{itemize}
    \item $\delta_D\geq
6\max\left\{
\sqrt{r_K \cdot x\|E\|},
\frac{\|E\|}{K}
\right\} $ (the assumption of Theorem \ref{theorem: main unified bound}),
\item $ K\|E\| \leq 2\sigma_1+\delta_D$ (see \eqref{additional assum: KE<2sigma+delta}),
\item $ \delta_D \leq \max \{\|E\|,\sigma_1\}$ and  $D \subset [\lambda_n-\|E\|, \lambda_1+\|E\|]$ (see \eqref{additional ass: detal<E+sigma}), 
\item $K > 10$ (see \eqref{additional assum: K>5}).
\end{itemize}

\subsection{Proof of \eqref{Est: AEA}} \label{subsec: proveinequalityF1} Without loss of generality, it suffices to consider $C_D=1$. Thus, $\Gamma$ is the rectangle with vertices
$$(x_0, H), (x_0, -H), (x_1, -H), (x_1, H), $$
where $x_0, x_1$ are the boundary points of $D$ and $H=2(\sigma_1+\|E\|+\delta_D)$. We shall prove that
   \begin{equation} \label{Est: AEACD1}
 \textstyle    G_1:=\frac{\int_{\Gamma}\|(z-A)^{-1} E (z-A)^{-1}\|\,\,\,  |dz|}{2\pi} \leq \frac{r_{K}\cdot x}{\delta_D} +  \frac{1}{K} \cdot \left[1+\frac{1}{\pi}+\frac{9}{\pi}\log \left( 4+\frac{4 \sigma_1}{\min\{\delta_D,\|E\|\}} \right) \right]  .   
    \end{equation}

First, we split $\Gamma$ into four segments: 
\begin{itemize}
\item $\Gamma_1:= \{ x_0+ \textbf{i}t : -H \leq t \leq H\}$, 
\item $\Gamma_2:= \{ x+ \textbf{i} H : x_0 \leq x \leq x_1\}$, 
\item $\Gamma_3:= \{ x_1+ \textbf{i} t : H \geq t \geq -H\}$,
\item $\Gamma_4:= \{ x - \textbf{i} H : x_1 \geq x \geq x_0\}$. 
\end{itemize}
The contour is illustrated below, with $\lambda, \lambda' \in D$ and $\lambda" \notin D$: 
\usetikzlibrary{decorations.pathreplacing}
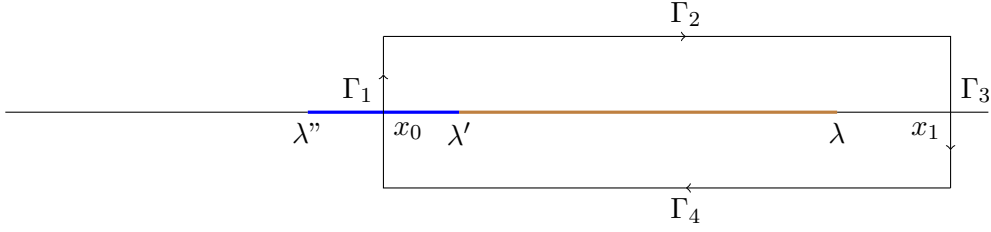
$$\begin{tikzpicture}
\coordinate (A) at (4,0);
\node[below] at (A){$\lambda"$};
\coordinate (A') at (6, 0);
\node[below] at (A'){$\lambda'$};
\coordinate (B) at (5,0);
\node[above left] at (B){$\Gamma_1$};
\node[below right] at (B) {$x_0$};
\coordinate (C) at (11,0);
\node[below] at (C){$\lambda$};
\coordinate (D) at (125mm,0);
\node[above right] at (D){$\Gamma_3$};
\node[below left] at (D){$x_1$};
\coordinate (E) at (13,0);
\coordinate(B') at (5,5mm);
\coordinate (F) at (9,1);
\node[above] at (F){$\Gamma_2$};
\coordinate (G) at (9,-1);
\node[below] at (G){$\Gamma_4$};
\draw (0,0) -- (A);
\draw[very thick,blue] (A) -- (A');
\draw[very thick, brown] (A')  -- (C);
\draw (C) -- (D) -- (E);
\draw[->] (B) -- (B');
\draw (B') -- (5,1);

\draw [->] (5,1) -- (9,1);
\draw (9,1) -- (125mm,1) -- (D);
\draw [->] (D) -- (125mm, -5mm);
\draw (125mm,-5mm) -- (125mm,-1);
\draw [->] (125mm,-1) -- (9,-1); 
\draw (9,-1) -- (5,-1) -- (B);

\end{tikzpicture}.$$
Therefore, 
$\textstyle 2\pi G_1 = \sum_{k=1}^4 N_k$, where we set 
$$ N_k:= \int_{\Gamma_k} \|(z-A)^{-1} E (z-A)^{-1} \| \, |dz|.$$

 We will repeatedly use the following elementary estimates.  
\begin{lemma} \label{ingegralcomputation1}
Let $a, H$ be positive numbers such that $a \leq H$. Then, 
\begin{equation*}
\textstyle \int_{-H}^{H} \frac{1}{t^2+a^2} dt \leq \frac{\pi}{a}.
\end{equation*}
\end{lemma}
\noindent\textit{Proof of Lemma \ref{ingegralcomputation1}.} We have 
\begin{equation*}
\begin{aligned}
\textstyle \int_{-H}^{H} \frac{1}{t^2+a^2} dt & = \textstyle 2 \int_{0}^H \frac{1}{t^2 +a^2} \textstyle = \frac{2}{a} \mathrm{arctan}(H/a) \leq \frac{2}{a} \cdot \frac{\pi}{2} =\frac{\pi}{a}.
\end{aligned}
\end{equation*}
\begin{lemma} \label{lem: tech0}
    Given \eqref{additional assum: KE<2sigma+delta}, \eqref{additional ass: detal<E+sigma}, and \eqref{additional assum: K>5}, we have 
    \[ \textstyle
    \frac{\sigma_1}{\min \{\delta_D, \|E\|\}} >3.
    \]
\end{lemma}
\noindent\textit{Proof of Lemma \ref{lem: tech0}.} If $\delta_D \leq \|E\|$, by \eqref{additional assum: KE<2sigma+delta} and \eqref{additional assum: K>5}, then 
\( 2\sigma_1+\delta_D \geq K\|E\|\geq 10\|E\| \geq 10\delta_D. 
\)
Therefore, 
\[\textstyle
\frac{\sigma_1}{\min \{\delta_D, \|E\|\}}=\frac{\sigma_1}{\delta_D} \geq \frac{9}{2} > 3.
\]
On the other hand, if $\delta_D > \|E\|$, by \eqref{additional ass: detal<E+sigma}, then 
\(
\delta_D \leq \sigma_1,
\)
and hence by \eqref{additional assum: KE<2sigma+delta} and \eqref{additional assum: K>5}, 
\(
3\sigma_1 \geq 2\sigma_1+\delta_D \geq K\|E\|\geq 10\|E\|.
\)
Therefore, 
\[\textstyle
\frac{\sigma_1}{\min \{\delta_D, \|E\|\}}=\frac{\sigma_1}{\|E\|} \geq \frac{10}{3} > 3.
\]

Back to our proof, we next bound $N_1, N_2, N_3, N_4$. 

\noindent { \bf \noindent Bounding $N_1$ and $N_3$.} Using the  decomposition  $(z-A)^{-1} = P+Q$ from \eqref{def: PQ}, we rewrite $N_1$ as 
\begin{equation}
\textstyle  N_1 = \int_{\Gamma_1} \big\| PEP+QEQ+(PEQ+QEP)\big\| \, |dz|.
\end{equation}
%
%
%
%
%
%
%
 Using the  triangle inequality, we obtain
\begin{equation}
\begin{split}
N_1 & \textstyle  \leq \int_{\Gamma_1} \big\| PEP\big\| \, |dz|+ \int_{\Gamma_1} \big\| QEQ\big\| \, |dz|+\int_{\Gamma_1} \big\| PEQ+QEP\big\| \, |dz|  \\
& = \textstyle\int_{\Gamma_1} \big\| \sum_{ i,j \in I_K(D)} \frac{1}{(z-\lambda_i)(z-\lambda_j)} u_i u_i^\top E u_j u_j^\top \big\| |dz| + \int_{\Gamma_1} \big\| \sum_{ i,j \notin I_K(D)} \frac{1}{(z-\lambda_i)(z-\lambda_j)} u_i u_i^\top E u_j u_j^\top \big\| |dz| \\
& + \textstyle \int_{\Gamma_1} \big\| \sum_{\substack{i \in I_K(D), j \notin I_K(D)  \\ \text{or}\, i \notin I_K(D), j \in I_K(D)}} \frac{1}{(z-\lambda_i)(z-\lambda_j)} u_i u_i^\top E u_j u_j^\top \big\| |dz|.
\end{split}
\end{equation}

\noindent For the first term, the orthonormality of $u_1, u_2, \dots, u_n$ gives
\begin{equation}
\begin{split}
\textstyle \int_{\Gamma_1} \big\| \sum_{i,j \in I_K(D)} \frac{1}{(z-\lambda_i)(z-\lambda_j)} u_i u_i^\top E u_j u_j^\top \big\| dz &\textstyle = \int_{\Gamma_1} \big\|\sum_{i,j \in I_K(D)}  \frac{ (u_i^\top E u_j)}{(z-\lambda_i)(z-\lambda_j)} u_i u_j^\top\big\|dz \\
& \leq  \textstyle  \int_{\Gamma_1} r_K \cdot \max_{i,j \in I_K(D)} \frac{|u_i^\top E u_j|}{\big|(z-\lambda_i)(z-\lambda_j)\big|} dz  \\
& \leq \textstyle r_{K}\cdot x \int_{\Gamma_1} \frac{1}{\min_{i,j \in I_K(D)}|(z-\lambda_i)(z-\lambda_j)|} dz.
\end{split}
\end{equation}
The last inequality follows from the definition of $x$.

Moreover, since $\Gamma_1:=\{ z\,|\,z= x_0 + \textbf{i} t, -H \leq t \leq H\}$ and 
\begin{equation} \label{x0lambdaiSmall}
    |x_0 - \lambda_i| \geq \delta_D \,\,\, \text{for all}\,\, i \in [n],
\end{equation}
we have
\[
\min_{i,j \in I_K(D)}|(z-\lambda_i)(z-\lambda_j)| = \min_{i,j \in I_K(D)} \sqrt{ (x_0 -\lambda_i)^2+t^2} \cdot \sqrt{(x_0 -\lambda_j)^2+t^2} \geq \delta_D^2+t^2. 
\]
Consequently, the right-hand side is at most
\begin{equation*}
    \begin{split}
    &\textstyle  r_{K}\cdot x \int_{-H}^{H} \frac{1}{t^2 +\delta_D^2} dt  \leq \frac{ \pi r_{K}\cdot x}{\delta_D} \,\,\, (\text{by Lemma \ref{ingegralcomputation1}}).    
    \end{split}
\end{equation*}

Next, we estimate the second term:
\begin{equation*}
\begin{split}
\textstyle\int_{\Gamma_1} \big\| \sum_{ i,j \notin I_K(D) } \frac{1}{(z-\lambda_i)(z-\lambda_j)} u_i u_i^\top E u_j u_j^\top\big\| |dz| & =\textstyle \int_{\Gamma_1} \big\| \left(\sum_{i \notin I_K(D)} \frac{u_iu_i^\top}{z- \lambda_i} \right) E \left( \sum_{i \notin I_K(D)} \frac{u_i u_i^\top}{z -\lambda_i} \right)\big\| |dz| \\
& \leq\textstyle\int_{\Gamma_1} \big\|\sum_{i \notin I_K(D)} \frac{u_iu_i^\top}{z- \lambda_i}\big\| \times \|E\| \times \big\|\sum_{i \notin I_K(D)} \frac{u_iu_i^\top}{z- \lambda_i} \big\| |dz|\\
& \leq \textstyle\int_{\Gamma_1} \frac{1}{\min_{i \notin I_K(D)} |z- \lambda_i|} \times\|E\| \times \frac{1}{\min_{i \notin I_K(D)} |z- \lambda_i|} |dz| \\
& = \textstyle\textstyle\|E\| \int_{\Gamma_1}  \frac{1}{ \min_{i \notin I_K(D)} |z-\lambda_i|^2} |dz| \\
& \leq \textstyle \|E\| \int_{-H}^{H} \frac{1}{\min_{i \notin I_K(D)} ((x_0-\lambda_i)^2+t^2)} dt. 
\end{split}
\end{equation*}
Since $i \notin  I_K(D)$, we have $|x_0 -\lambda_i| \geq K\|E\|.$ 
Hence,
\begin{equation*}
    \begin{split}
\textstyle \int_{\Gamma_1} \big\| \sum_{i,j\notin I_K(D)} \frac{1}{(z-\lambda_i)(z-\lambda_j)} u_i u_i^\top E u_j u_j^\top \big\| |dz| & \leq \textstyle \|E\| \int_{-H}^{H} \frac{1}{t^2+ (K\|E\|)^2 } dt \\
 &  \leq \textstyle \frac{\pi \|E\|}{K\|E\|} = \frac{\pi}{K} \,\,\,(\text{by Lemma \ref{ingegralcomputation1}}).
    \end{split}
\end{equation*}

\noindent Finally, we consider the last term:
\begin{equation} \label{lastterm1}
\begin{split}
 \textstyle\int_{\Gamma_1} \big\| \sum_{\substack{i \in I_K(D), j \notin I_K(D)  \\ \text{or}\, i \notin I_K(D), j \in I_K(D)}} \frac{1}{(z-\lambda_i)(z-\lambda_j)} u_i u_i^\top E u_j u_j^\top \big\| |dz| & \textstyle\leq 2 \int_{\Gamma_1} \big\| \sum_{i \in I_K(D)} \frac{u_iu_i^\top}{z- \lambda_i} \big\| \cdot \|E\| \cdot \big\|\sum_{j \notin I_K(D)} \frac{u_j u_j^\top}{z- \lambda_j} \big\| |dz|\\
& \leq\textstyle 2 \int_{\Gamma_1} \frac{1}{\min_{i \in I_K(D)} |z- \lambda_i|} \times \|E\| \times \frac{1}{\min_{j \notin I_K(D)} |z- \lambda_j|} |dz| \\ 
  & =\textstyle 2 \|E\| \int_{\Gamma_1}  \frac{1}{\min_{i \in I_K(D), j \notin I_K(D)} \mid (z-\lambda_i)(z-\lambda_j)\mid} |dz|.
   \end{split}
\end{equation}
Since $|z -\lambda_i| \geq \delta_D$ for all $i \in I_K(D)$ and $|z -\lambda_j| \geq K\|E\|$ for all $j \notin I_K(D)$, the right-hand side is at most
\begin{equation*}
    \begin{split}
    & \textstyle 2 \|E\| \int_{-H}^{H} \frac{1}{\sqrt{(t^2+\delta_D^2)(t^2+(K\|E\|)^2)}} dt = 4 \|E\| \int_{0}^{H}  \frac{1}{\sqrt{(t^2+\delta_D^2)(t^2+(K\|E\|)^2)}} dt.     
    \end{split}
\end{equation*}
If $\delta_D \leq \|E\|$, using $\sqrt{t^2+a^2} \geq\frac{t+a}{\sqrt{2}}$, we obtain 
\begin{equation*}
\begin{split}
\textstyle\int_{0}^{H}  \frac{dt}{\sqrt{(t^2+\delta_D^2)(t^2+(K\|E\|)^2)}}  & \textstyle\leq \int_{0}^{H} \frac{2}{(t+\delta_D)(t+K\|E\|)} dt \\
& =\textstyle\frac{2}{K\|E\| -\delta_D} \int_{0}^{H} \left(\frac{1}{t+\delta_D} -\frac{1}{t+ K\|E\|} \right)dt \\
& = \textstyle\frac{2}{K\|E\| -\delta_D} \left[ \log \left( \frac{H+\delta_D}{\delta_D} \right) - \log \left( \frac{H+K\|E\|}{K\|E\|} \right)  \right]\\
& \leq\textstyle \frac{2}{\|E\| (K-1)} \times \log \left( \frac{H+\delta_D}{\delta_D} \right).
\end{split}
\end{equation*}
 Moreover, by \eqref{additional assum: K>5} and \eqref{additional assum: KE<2sigma+delta}, we further have $K \geq 10$ and $2 \sigma_1+\delta_D \geq K\|E\| \geq 10\|E\| \geq \delta_D+9\|E\|$. Hence, $\sigma_1 \geq 9\|E\|/2>2\|E\|$, yielding
 \[ \textstyle
 \frac{H+\delta_D}{\delta_D}=3+ \frac{2 \sigma_1+2\|E\|}{\delta_D} \leq 3+\frac{3 \sigma_1}{\delta_D}. 
 \]
 Consequently, 
 \begin{equation}  \label{lasterm2}
\textstyle\int_{0}^{H}  \frac{dt}{\sqrt{(t^2+\delta_D^2)(t^2+(K\|E\|)^2)}}  \leq    \frac{2}{\|E\| (K-1)} \times \log \left( \frac{H+\delta_D}{\delta_D} \right) \leq   \frac{2}{\|E\| (K-1)} \times \log \left(3+ \frac{3\sigma_1}{\delta_D} \right) 
 \end{equation}

If $\delta_D > \|E\|$, similarly 
\begin{equation*}
\begin{split}
\textstyle\int_{0}^{H}  \frac{dt}{\sqrt{(t^2+\delta_D^2)(t^2+(K\|E\|)^2)}}   \textstyle &\leq \textstyle\int_{0}^{H}  \frac{dt}{\sqrt{(t^2+\|E\|^2)(t^2+(K\|E\|)^2)}} \leq\textstyle \frac{2}{\|E\| (K-1)} \times \log \left( \frac{H+\|E\|}{\|E\|} \right).
\end{split}
\end{equation*}
By \eqref{additional ass: detal<E+sigma}, we have $\delta_D \leq\max\{\sigma_1,\|E\|\}$. This implies 
\begin{equation}  \label{lasterm2.1}
   \textstyle\int_{0}^{H}  \frac{dt}{\sqrt{(t^2+\delta_D^2)(t^2+(K\|E\|)^2)}}   \textstyle \leq   \frac{2}{\|E\| (K-1)} \times \log \left( \frac{H+\|E\|}{\|E\|} \right) \leq \frac{2}{\|E\| (K-1)} \times \log \left(4+ \frac{4 \sigma_1}{\|E\|} \right) 
\end{equation}

Together, \eqref{lastterm1}, \eqref{lasterm2}, and \eqref{lasterm2.1} imply that the last term is at most 
$$\textstyle \frac{8}{K-1} \times \log \left( 4+\frac{4 \sigma_1}{\min\{\delta_D,\|E\|\}} \right) \leq \frac{9}{K} \times \log \left( 4+\frac{4 \sigma_1}{\min\{\delta_D,\|E\|\}} \right) (\,\text{since $K \geq 10$}).$$

These estimates imply that 
\begin{equation} \label{EAEM1}
\textstyle N_1 \leq \frac{\pi r_K \cdot x}{\delta_D} + \frac{\pi }{K} + \frac{9}{K} \times \log \left( 4+\frac{4 \sigma_1}{\min\{\delta_D,\|E\|\}} \right)  .
\end{equation}

The same argument, with $x_0$ replaced by $x_1$, gives
\begin{equation} \label{EAEM3}
   \textstyle  N_3 \leq \frac{\pi r_K \cdot x}{\delta_D} + \frac{\pi }{K} + \frac{9}{K} \times \log \left( 4+\frac{4 \sigma_1}{\min\{\delta_D,\|E\|\}} \right).
\end{equation}
\noindent \textbf{Bounding $N_2$ and $N_4$}. On the horizontal segments of $\Gamma$, we  use the resolvent estimate that 
$$\textstyle \|(z-A)^{-1} E (z-A)^{-1}\| \leq \frac{\|E\|}{\min_{i \in [n]}|z- \lambda_i|^2}.$$
Therefore, 
\begin{equation} \label{F_1f1inequality1}
\textstyle N_2 \leq \int_{\Gamma_2}  \frac{1}{\min_{i \in [n]} |z-\lambda_i|^2 } \|E\|\, |dz| = \|E\| \int_{\Gamma_2}  \frac{1}{\min_{i \in [n]} |z-\lambda_i|^2 } |dz|.
\end{equation}
Moreover, since $\Gamma_2:= \{ z \,|\, z = x+ \textbf{i} H, x_0 \leq x \leq x_1 \},$
\begin{equation} \label{F_1f1inequality3}
\begin{split}
&\textstyle \int_{\Gamma_2}   \frac{1}{\min_{i \in [n]} |z-\lambda_i|^2 } dz = \int_{x_0}^{x_1}  \frac{1}{ \min_{i \in [n]} ((x-\lambda_i)^2+H^2)} dx \leq \int_{x_0}^{x_1} \frac{1}{H^2} dx = \frac{|x_1 - x_0|}{H^2}.
\end{split}
\end{equation}
Together, \eqref{F_1f1inequality1} and \eqref{F_1f1inequality3} imply 
\begin{equation} \label{N2bound}
  N_2 \leq \frac{\|E\| \cdot |x_1-x_0|}{H^2}.  
\end{equation}
Similarly,
\begin{equation} \label{N4bound}
    N_4 \leq \frac{\|E\| \cdot |x_1-x_0|}{H^2}.
\end{equation}
Combining \eqref{EAEM1}, \eqref{EAEM3}, \eqref{N2bound}, and \eqref{N4bound}, we obtain 
\begin{equation} \label{G1bound0}
    \textstyle G_1 \leq \frac{\sum_{i=1}^4 N_i }{2 \pi} \leq \frac{r_K \cdot x}{\delta_D}+ \frac{1}{K}+\frac{9/\pi}{K} \cdot \log \left( 4+\frac{4 \sigma_1}{\min\{\delta_D,\|E\|\}} \right) + \frac{\|E\|\cdot|x_1 -x_0|}{ \pi H^2}.
\end{equation}
%
%
Since 
\(
\textstyle H=2 (\sigma_1+\|E\|+\delta_D) \geq K\|E\|
\)
and 
\(
|x_1 - x_0| \leq 2 \sigma_1+2\|E\| < H
\), we have
\(
\textstyle \frac{\|E\|\cdot|x_1 -x_0|}{ \pi H^2} \leq \frac{\|E\|}{\pi H} \leq \frac{1}{\pi K}.
\)
Substituting this estimate into \eqref{G1bound0}, we prove \eqref{Est: AEACD1}, and hence \eqref{Est: AEA}. 
\subsection{Proof of \eqref{Est: zAinverse}} As in the proof of \eqref{Est: AEA},  we assume, without loss of generality, that $C_D=1$. In this setting, $\Gamma$ is a rectangle with vertices
$$(x_0, H), (x_0, -H), (x_1, -H), (x_1, H),$$
where $x_0, x_1$ are the boundary points of $D$ and $H=2(\sigma_1+\|E\|+\delta_D)$. Thus, it is sufficient to prove that 
 \begin{equation} \label{Est: zAinverseCD1}
    \textstyle G_2:=\frac{\int_{\Gamma}\|(z-A)^{-1} \|\,\,\,  |dz|}{2\pi} \leq \frac{2+ 4 \log \left(9+\frac{5\sigma_1}{\min\{\delta_D,\|E\|\}} \right)}{2 \pi} .   
    \end{equation}

%
 %
  
We decompose $\Gamma$ into $\Gamma_1, \Gamma_2, \Gamma_3, \Gamma_4$ as in the previous subsection, and hence
\[
G_2 = \sum_{i=1}^4 \frac{\int_{\Gamma_i}\|(z-A)^{-1} \|\,\,\,  |dz|}{2\pi}.
\]
On the vertical segment $\Gamma_1$, we have 
\[
\textstyle \int_{\Gamma_1}\|(z-A)^{-1}\|\, |dz| = \int_{-H}^{H} \frac{1}{ \min_{i \in [n]}\sqrt{(x_0 -\lambda_i)^2+t^2}} dt \leq \int_{-H}^{H} \frac{1}{ \sqrt{\delta_D^2+t^2}} dt = 2 \int_{0}^H  \frac{1}{ \sqrt{\delta_D^2+t^2}} dt.
\]
Moreover, 
$$\textstyle \int_{0}^{H} \frac{1}{ \sqrt{\delta_D^2+t^2}} dt = \log \left(\frac{H + \sqrt{\delta_D^2+H^2}}{\delta_D} \right) \leq \log \left(\frac{2H+\delta_D }{\delta_D} \right) \leq \log \left(9+\frac{5\sigma_1}{\min\{\delta_D,\|E\|\}} \right).$$
Therefore, 
\begin{equation} \label{ieqGamma1}
    \textstyle \int_{\Gamma_1}\|(z-A)^{-1}\|\, |dz| \leq 2\log \left(9+\frac{5\sigma_1}{\min\{\delta_D,\|E\|\}} \right).
\end{equation}
Similarly, we  obtain 
\begin{equation} \label{ieqGamma3}
  \textstyle  \int_{\Gamma_3}\|(z-A)^{-1}\|\, |dz| = \int_{-H}^{H} \frac{1}{ \min_{i \in [n]}\sqrt{(x_1 -\lambda_i)^2+t^2}} dt \leq  2\log \left(9+\frac{5\sigma_1}{\min\{\delta_D,\|E\|\}} \right).
\end{equation}

On the horizontal segment $\Gamma_2$, we have
\begin{equation}\label{ieqGamma2}
    \textstyle \int_{\Gamma_2}\|(z-A)^{-1}\|\, |dz| = \int_{x_0}^{x_1} \frac{1}{ \min_{i \in [n]}\sqrt{(t -\lambda_i)^2+H^2}} dt \leq \frac{|x_1 -x_0|}{H} \leq \frac{2\sigma_1+2\|E\|}{H} \leq 1.
\end{equation}
Similarly,
\begin{equation} \label{ieqGamma4}
  \textstyle  \int_{\Gamma_4}\|(z-A)^{-1}\|\, |dz| \leq 1.
\end{equation}
Together, \eqref{ieqGamma1}, \eqref{ieqGamma3}, \eqref{ieqGamma2}, and \eqref{ieqGamma4} prove \eqref{Est: zAinverseCD1}, and hence \eqref{Est: zAinverse}.

\section{Proof of Theorem \ref{cor: simpleWigeruplow}} \label{sec: proof of deformWig}
In this section, we prove Theorem~\ref{cor: simpleWigeruplow}, which bounds \(|\tilde\lambda_p-\lambda_p|\) when \(A\) is symmetric with
\(
\max\{\operatorname{rank}A,\log\|A\|\}\leq \log^c n,
\)
and \(E\) is a sub-Gaussian matrix.

Without loss of generality, assume that \(\lambda_p>0\). The case \(\lambda_p<0\) follows by applying the same argument to \(-A-E\), with the index \(p\) replaced by \(n+1-p\). 

Set
\(
K:=\frac{\lambda_p}{2\|E\|},
\)
so that $K\|E\|=\frac{\lambda_p}{2}$, and hence with probability \(1-o(1)\),
$$ \textstyle
\frac{\sqrt{n}}{K} = \frac{2\|E\| \cdot \sqrt{n}}{\lambda_p} < \frac{6n}{\lambda_p}.
$$

We first derive an upper bound for \(\tilde\lambda_p\). Applying Corollary~\ref{cor: lambdapWignerupper} with this choice of \(K\), we obtain that if
$$
\delta_{p-1}
\geq
270
\log\left(3+\sigma_1\right) \textstyle
\max\left\{
\frac{6n}{\lambda_p},\,
\sqrt{r_{p-1,K}\log n}\,\cdot n^{1/4}
\right\},
$$
then, with probability \(1-o(1)\),
$$
\tilde\lambda_p
\leq
\lambda_p+
135
\log\left(3+\sigma_1\right) \textstyle
\max\left\{
\frac{6n}{\lambda_p},\,
\sqrt{r_{p-1,K}\log n}\,\cdot n^{1/4}
\right\}.
$$

By definition, \(r_{p-1,K}\) counts the eigenvalues at least
$$ \textstyle
\lambda_{p-1}-\frac{\delta_{p-1}}{2}-K\|E\|
=
\frac{\lambda_{p-1}}{2}>0.
$$
Hence,
$$
r_{p-1,K}\leq \operatorname{rank}A\leq\log^c n.
$$
Together with \(\log\sigma_1\leq\log^c n\), this shows that there exist constants \(C_1,C_2>0\), depending only on \(c\), such that
\begin{equation} \label{upperDeltap-1}
    \text{if}\,\delta_{p-1}
\geq
\log^{C_1}n \textstyle
\max\left\{\frac{n}{\lambda_p},n^{1/4}\right\}, \,\text{then}\, \tilde\lambda_p
\leq
\lambda_p+
\log^{C_2}n
\max\left\{\frac{n}{\lambda_p},n^{1/4}\right\}.
\end{equation}

We next derive a lower bound for \(\tilde\lambda_p\), which requires a more delicate argument. We split the proof into two cases.

\medskip
\noindent\textit{Case 1: \(\lambda_{p+1}>0\).}
Applying Corollary~\ref{cor: lambdaPWigner} with the same choice of \(K\), if
$$
\delta_p
\geq
270
\log\left(3+\sigma_1\right) \textstyle
\max\left\{
\frac{6n}{\lambda_p},\,
\sqrt{r_{p,K}\log n}\,\cdot n^{1/4}
\right\},
$$
then, with probability \(1-o(1)\),
$$
\tilde\lambda_p
\geq
\lambda_p-
135
\log\left(3+\sigma_1\right) \textstyle
\max\left\{
\frac{6n}{\lambda_p},\,
\sqrt{r_{p,K}\log n}\,\cdot n^{1/4}
\right\}.
$$

Since \(\lambda_{p+1}>0\), we have
\(
\lambda_p-\frac{\delta_p}{2}-K\|E\|
=
\frac{\lambda_{p+1}}{2}>0.
\)
Thus, since $r_{p,K}$ counts the eigenvalues at least $\lambda_p-\frac{\delta_p}{2}-K\|E\|$, we have
$$
r_{p,K}\leq \operatorname{rank}A\leq\log^c n.
$$
Consequently, there exist constants \(C_3,C_4>0\), depending only on \(c\), such that
\begin{equation} \label{lowerDeltap}
    \text{if}\,\delta_p
\geq
\log^{C_3}n \textstyle
\max\left\{\frac{n}{\lambda_p},n^{1/4}\right\}, \,\text{then}\, \tilde\lambda_p
\geq
\lambda_p-
\log^{C_4}n
\max\left\{\frac{n}{\lambda_p},n^{1/4}\right\}.
\end{equation}

\noindent\textit{Case 2: \(\lambda_{p+1}\leq0\).}
In this case, \(\delta_p\geq\lambda_p\). We repeat the argument preceding Theorem~\ref{theo: leadinglower}. Consider
\(
D=(\lambda_p-T,+\infty),
\)
with 
\(
T<\frac{\lambda_p}{4}\leq\frac{\delta_p}{4}.
\)
Then
$$
\delta_D=\min\{T,\delta_p-T\}=T.
$$
Let \(r'_K\) denote the number of eigenvalues of \(A\) that are at least
\(
\lambda_p-\frac{\lambda_p}{4}-K\|E\|.
\)
Set
$$
T:= \textstyle
45
\log\left(3+\sigma_1\right)
\max\left\{
\frac{\|E\|}{K},
\sqrt{r'_Kx\|E\|}
\right\}.
$$
If
$$ \textstyle
\lambda_p
\geq
180
\log\left(3+\sigma_1\right)
\max\left\{
\frac{\|E\|}{K},
\sqrt{r'_Kx\|E\|}
\right\},
$$
then \(T\leq\lambda_p/4\). Moreover, the number of eigenvalues within distance \(K\|E\|\) of \(D\) is at most \(r'_K\). Hence, Theorem~\ref{theo: main0} implies that \(D\) is stable, and therefore
$$ \textstyle
\tilde\lambda_p
\geq
\lambda_p-T
=
\lambda_p-
45
\log\left(3+\sigma_1\right)
\max\left\{\frac{\|E\|}{K},
\sqrt{r'_Kx\|E\|}
\right\}.
$$
For our choice \(K=\lambda_p/(2\|E\|)\), we have
\(
\lambda_p-\frac{\lambda_p}{4}-K\|E\|
=\frac{\lambda_p}{4}>0.
\)
Thus,
$$
r'_K\leq\operatorname{rank}(A)\leq\log^c n.
$$

Furthermore, by Lemma~\ref{Wigner property}, with probability \(1-o(1)\),
$$
x=O(\sqrt{\log n}),
\qquad
\|E\|\leq (2+o(1))\sqrt n.
$$
Together with \(\log\sigma_1\leq\log^c n\), this shows that there exist constants \(C_5,C_6>0\), depending only on \(c\), such that
\begin{equation}\label{lowerLamdap}
\text{if }\quad
\lambda_p
\geq
\log^{C_5}n \textstyle
\max\big\{\frac{n}{\lambda_p},n^{1/4}\big\},
\quad\text{then}\quad
\tilde\lambda_p
\geq
\lambda_p-
\log^{C_6}n
\max\big\{\frac{n}{\lambda_p},n^{1/4}\big\}.
\end{equation}

Combining \eqref{upperDeltap-1}, \eqref{lowerDeltap}, and \eqref{lowerLamdap}, and taking
$$
c_1:=\max\{C_1,C_3,C_5\},
\qquad
c_2:=\max\{C_2,C_4,C_6\},
$$
we conclude that if
$$
\min\{\delta_{(p)},\lambda_p\}
\geq
\log^{c_1}n \textstyle
\max\left\{\frac{n}{\lambda_p},n^{1/4}\right\},
$$
then, with probability \(1-o(1)\),
$$
|\tilde\lambda_p-\lambda_p|
\leq
\log^{c_2}n \textstyle
\max\left\{\frac{n}{\lambda_p},n^{1/4}\right\}.
$$
Similarly, for $\lambda_p < 0$, we also have that if
$$
\min\{\delta_{(p)},|\lambda_p|\}
\geq
\log^{c_1}n \textstyle
\max\left\{\frac{n}{|\lambda_p}|,n^{1/4}\right\},
$$
then, with probability \(1-o(1)\),
$$
|\tilde\lambda_p-\lambda_p|
\leq \textstyle
\log^{c_2}n
\max\left\{\frac{n}{|\lambda_p|},n^{1/4}\right\}.
$$
These claims prove Theorem~\ref{cor: simpleWigeruplow}.

\section{An application of relative perturbation bound}\label{secion: app comparision relative}


The relative bound discussed in Subsection \ref{subsubsec: relative bound} shows that if $A$ is positive definite and
\[
\bigl\|A^{-1/2}E A^{-1/2}\bigr\| < 1, 
\]
 then
\[
\frac{|\tilde{\lambda}_n-\lambda_n|}{\lambda_n}
\leq
\bigl\|A^{-1/2}E A^{-1/2}\bigr\|.
\]
Under the setting of Theorem \ref{theo: least2}, the right-hand side can be further bounded as
\[ \textstyle
\bigl\|A^{-1/2}E A^{-1/2}\bigr\|
\leq
\frac{rx}{\lambda_n}
+
2\sqrt{\frac{\|E\|}{K\lambda_n}}
+
\frac{1}{K}.
\]


Therefore,  if $\frac{rx}{\lambda_n}
+
2\sqrt{\frac{\|E\|}{K\lambda_n}}
+
\frac{1}{K}< 1$, then
\[ \textstyle
\lambda_n -\tilde{\lambda}_n \leq rx+ 2 \sqrt{\lambda_n \cdot \frac{\|E\|}{K}} + \frac{\lambda_n}{K}. 
\]
Let us focus on the regime in which the right-hand side improves upon Weyl’s inequality. This regime requires that $\lambda_n/K <\|E\|$ and $rx \leq \|E\|$. The bound, then, simplifies to 
\begin{equation} \label{boundRelative}
    \textstyle
\lambda_n -\tilde{\lambda}_n \leq rx+ 2 \sqrt{\lambda_n \cdot \frac{\|E\|}{K}} =O\left(rx+  \sqrt{\lambda_n \cdot \frac{\|E\|}{K}} \right).
\end{equation}
Let us compare \eqref{boundRelative} with our bound in Theorem~\ref{theo: least2}. When \(A\) is positive definite, Theorem~\ref{theo: least2} gives 
 \[ \textstyle
   \lambda_n- \tilde{\lambda}_n =\tilde{O} \left(
\max\bigg\{
\frac{\|E\|}{K},\sqrt{rx\|E\|}
\bigg\} \right).
    \]
Our bound is sharper when
\[
rx\le \frac{\lambda_n}{K}.
\]
This occurs when the interaction between $E$ and the eigenvectors of $A$ is weak ($x$ is small), and the spectrum of $A$ is not clustered near $\lambda_n$. For example, this holds when $E$ is random, $\lambda_n=n^{1/3}$, and fewer than $n^{1/6}$ eigenvalues of $A$ are below $n^{2/3}$. 

Our bound is weaker when 
$\sqrt{\|E\| \cdot \frac{\lambda_n}{K}} < rx < \|E\|$.

\end{document}